\documentclass[11pt]{amsart}

\usepackage{fullpage,amsmath,amsthm,amsfonts,amssymb,
graphicx,amscd,float,hyperref,enumitem,setspace,amsaddr}
\usepackage{comment}
\usepackage[T1]{fontenc}
\usepackage[utf8]{inputenc}
\usepackage{mathtools}
\usepackage{aliascnt}
\usepackage{tikz}
\usetikzlibrary{arrows.meta,patterns}

\hypersetup{colorlinks=true,linkcolor=blue,citecolor=blue,urlcolor=blue}

\newtheorem{thm}{Theorem}[section]

\newaliascnt{lem}{thm}
\newtheorem{lem}[lem]{Lemma}
\aliascntresetthe{lem}

\newaliascnt{qst}{thm}

\aliascntresetthe{qst}

\newaliascnt{cor}{thm}
\newtheorem{cor}[cor]{Corollary}
\aliascntresetthe{cor}

\newaliascnt{prop}{thm}
\newtheorem{prop}[prop]{Proposition}
\aliascntresetthe{prop}

\newaliascnt{ex}{thm}

\aliascntresetthe{ex}

\newaliascnt{nt}{thm}

\aliascntresetthe{nt}

\newaliascnt{cl}{thm}

\aliascntresetthe{cl}

\theoremstyle{definition}
\newaliascnt{rem}{thm}
\newtheorem{rem}[rem]{Remark}
\aliascntresetthe{rem}

\newaliascnt{conv}{thm}
\newtheorem{conv}[conv]{Convention}
\aliascntresetthe{conv}

\newaliascnt{defn}{thm}
\newtheorem{defn}[defn]{Definition}
\aliascntresetthe{defn}

\newaliascnt{pd}{thm}

\aliascntresetthe{pd}

\newaliascnt{prob}{thm}

\aliascntresetthe{prob}

\usepackage[nameinlink,capitalize]{cleveref}
\crefname{thm}{Theorem}{Theorems}
\Crefname{thm}{Theorem}{Theorems}
\crefname{prop}{Proposition}{Propositions}
\Crefname{prop}{Proposition}{Propositions}
\crefname{lem}{Lemma}{Lemmas}
\Crefname{lem}{Lemma}{Lemmas}
\crefname{cor}{Corollary}{Corollaries}
\Crefname{cor}{Corollary}{Corollaries}
\crefname{rem}{Remark}{Remarks}
\Crefname{rem}{Remark}{Remarks}
\crefname{ex}{Example}{Examples}
\Crefname{ex}{Example}{Examples}
\crefname{defn}{Definition}{Definitions}
\Crefname{defn}{Definition}{Definitions}
\crefname{prob}{Problem}{Problems}
\Crefname{prob}{Problem}{Problems}
\crefname{qst}{Question}{Questions}
\Crefname{qst}{Question}{Questions}
\crefname{cl}{Claim}{Claims}
\Crefname{cl}{Claim}{Claims}
\crefname{nt}{Notation}{Notations}
\Crefname{nt}{Notation}{Notations}
\crefname{conv}{Convention}{Conventions}
\Crefname{conv}{Convention}{Conventions}
\crefname{pd}{Proposition-Definition}{Proposition-Definitions}
\Crefname{pd}{Proposition-Definition}{Proposition-Definitions}

\newcommand{\F}{F(a,b)}

\newcommand{\Z}{\mathbb Z}
\newcommand{\N}{\mathbb N}

\renewcommand{\phi}{\varphi}

\newcommand{\R}{\mathbb R}
\newcommand{\Pcal}{\mathcal P}
\newcommand{\Rcal}{\mathcal R}
\newcommand{\Area}{\operatorname{Area}}

\newcommand{\ncl}{\operatorname{ncl}}

\newcommand{\eps}{\varepsilon}
\newcommand{\Nzero}{\N\cup\{0\}}
\newcommand{\Qcal}{\mathcal Q}

\newcommand{\Acal}{\mathcal A}

\newcommand{\fpreceq}{\preceq_{\mathrm f}}
\newcommand{\fsimeq}{\simeq_{\mathrm f}}

\makeatletter
\@namedef{subjclassname@2020}{%
  \textup{2020} Mathematics Subject Classification}
\makeatother

\title[Dehn functions for infinite presentations]{On Dehn functions for infinite group presentations}

\author{Ilya Kapovich}

\address{Department of Mathematics and Statistics, Hunter College of CUNY\newline
  \indent 695 Park Ave, New York, NY 10065, U.S.A.
  \newline \indent e-mail {\tt ik535@hunter.cuny.edu}
  \newline\indent  ORCID 0000-0002-7694-6236
  }

\date{}

\begin{document}

\begin{abstract}
We study Dehn functions of presentations with finite generating sets and possibly infinite defining relator sets, focusing first on finitely presentable groups.  Unlike finite presentations, such presentations can exhibit new filling behavior even for groups as classical as $\Z^2$ and torsion-free small-cancellation groups, including closed surface groups.

For $\Z^2$ on the standard generators $a,b$, we prove that the presentation
\[
  \left\langle a,b\ \middle|\ [a^{2^k},b],\ k=0,1,2,\ldots\right\rangle
\]
has Dehn function of order $n\log n$.  

We also construct infinite presentations of $\Z^2$ on the same generators with logarithmic Dehn function in a fine asymptotic sense, and with polynomial-envelope behavior below the quadratic scale: for every $0<\alpha<2$, there is such a presentation whose Dehn function satisfies a global upper bound $O(n^\alpha)$ and has matching $n^\alpha$-order lower-bound peaks along an infinite sequence of lengths.  These examples give continuum many distinct fine filling-growth types for presentations of $\Z^2$ on the fixed generating set $\{a,b\}$.  For finite $C'(1/6)$ small-cancellation presentations with no proper-power relators, we obtain analogous logarithmic and polynomial-envelope examples with exponents $0<\alpha<1$; the standard presentations of closed orientable surface groups of genus at least two are included in this class.

\end{abstract}

\maketitle

\tableofcontents

\section{Introduction}\label{sec:introduction}

Dehn functions, or isoperimetric functions of finite group presentations, are
central in geometric group theory.  They measure the filling area needed to
prove that a word in the generators represents the identity, equivalently the
minimum number of faces in a van Kampen diagram with a prescribed boundary
label.  They connect presentations, large-scale geometry, filling inequalities,
and algorithmic properties of groups.  This viewpoint goes back to classical
combinatorial group theory and Dehn's solution of the word problem for surface
groups; standard references include Lyndon--Schupp and Bridson's survey
\cite{LyndonSchupp,BridsonWordProblem}.

For finite presentations, Dehn functions have robust geometric meaning.  Their
growth type, up to the usual equivalence relation, is a quasi-isometry invariant;
see Alonso and Gersten, and the treatments in Epstein et al.\ and
Bridson--Haefliger
\cite{AlonsoQI,GerstenIso,EpsteinWordProcessing,BridsonHaefliger}.  Gromov
proved that a finitely presented group is word-hyperbolic exactly when it
satisfies a linear isoperimetric inequality \cite{GromovHyperbolic}.  Gersten
proved the related algorithmic characterization: for a finitely presented group,
solvability of the word problem is equivalent to the existence, for some,
equivalently every, finite presentation, of a recursive Dehn function
\cite{GerstenIso}.

The range of possible growth types for Dehn functions of finitely presented
groups has been studied extensively.  Work of Birget--Ol'shanskii--Rips--Sapir,
Sapir--Birget--Rips, and Ol'shanskii--Sapir relates Dehn functions to
computational complexity and gives broad constructions realizing prescribed or
highly controlled growth types
\cite{BBR,SBR,OlshanskiiSapirLengthArea}.  Work of Noel Brady, Bridson, and
others produced finitely presented groups whose Dehn functions grow like powers
$n^\alpha$ for many non-integral exponents $\alpha$
\cite{BradyBridsonGap}.  At the low end, however, finite presentations exhibit a
rigid gap: a finitely presented group with subquadratic Dehn function is
word-hyperbolic, and therefore has linear Dehn function; see Gromov,
Ol'shanskii, Bowditch, and Brady--Bridson
\cite{GromovHyperbolic,OlshanskiiSubquadratic,BowditchSubquadratic,BradyBridsonGap}.
Thus no genuine growth strictly between $n$ and $n^2$ occurs for finite
presentations, up to the usual equivalence.
There is also a simple cardinality limitation: there are only countably many
finite group presentations, and hence only countably many growth types of Dehn
functions arising from finite presentations.  By contrast, for presentations
with a finite generating set but an infinite set of defining relators, the
number of possible growth types can have the cardinality of the continuum even
for a fixed group with a fixed generating set.  The results below demonstrate
this point explicitly: for the fixed group $\Z^2$ on the fixed generating set
$\{a,b\}$, \cref{cor:z2-polynomial-from-general} produces a family
$\{\Pcal_\alpha:0<\alpha<2\}$ with global upper bounds and matching
lower-bound peaks at the corresponding exponents.  Moreover, part~(2) of
\cref{cor:z2-polynomial-from-general} says that these Dehn functions are
pairwise not finely equivalent.  Thus, in the finer comparison sense used in
this paper, the presentations give continuum many distinct filling-growth
behaviors.

Infinite presentations behave differently.  Their Dehn functions depend on the
chosen presentation, since very long unit-cost relators can change filling area
dramatically.  Still, infinite presentations and relative Dehn functions are
central in geometric group theory.  Osin's approach to relatively hyperbolic
groups uses relative presentations encoding peripheral subgroups
\cite{OsinRelHyp}.  Grigorchuk and Ivanov studied infinite-presentation Dehn
functions from another viewpoint, introducing variants that minimize vertices or
edges rather than faces in van Kampen diagrams \cite{GrigorchukIvanov}.
Cummins investigated algorithmic aspects for decidable infinite
presentations~\cite{Cummins}.  Osin and Rybak introduced a canonical
alternative, Dehn spectra, by allowing all null relators up to a length scale;
their construction gives quasi-isometry invariants of finitely generated
groups~\cite{OsinRybak}.  Here we take the complementary
presentation-dependent viewpoint: the group and finite generating set are fixed,
while the infinite relator set is chosen deliberately.

This paper studies low-degree filling behavior with finite generating sets and infinite defining relator sets.  If $G$ is finitely presentable and
$\mathcal P=\langle X\mid R\rangle$ is a presentation of $G$ with $X$ finite
and $R$ possibly infinite, then there is a finite subset $R_0\subseteq R$ such
that $\mathcal P_0=\langle X\mid R_0\rangle$ also presents $G$.  Hence the
Dehn function of $\mathcal P$ is bounded above by the Dehn function of the
finite presentation $\mathcal P_0$ of $G$.

The paper focuses on classical groups, especially free abelian and
small-cancellation groups; standard closed surface groups occur as examples.
Each defining relator has area cost one, regardless of length.  This convention
separates filling area from relator length and from the algorithmic complexity of
recognizing defining relators.  The resulting function is presentation-dependent,
not a quasi-isometry invariant.  The goal is to identify subquadratic phenomena,
impossible for finite presentations, that nevertheless occur for familiar groups
under this convention.

Unless explicitly stated otherwise, a presentation in this paper has the form
$\Pcal=\langle X\mid \Rcal\rangle$, where $X$ is finite and $\Rcal$ is an
arbitrary, possibly infinite, set of cyclically reduced words in the free group
$F(X)$.  It defines the group
\[
  G(\Pcal)=F(X)/\ncl_{F(X)}(\Rcal).
\]
The associated Dehn function counts the minimum number of defining relator cells
in a van Kampen diagram; every relator in $\Rcal$ has unit cost.  This
convention is the usual area convention for finite presentations, but here
$\Rcal$ may be infinite.  Because the usual coarse comparison identifies every bounded or sublinear
function with the linear function, logarithmic and sublinear growth are compared
using the finer domination relation introduced in \cref{def:fine-comparison}.

The first main result constructs an explicit infinite presentation of $\Z^2$
whose Dehn function lies in the finite-presentation subquadratic gap.

\begin{thm}[A dyadic-strip presentation of $\Z^2$]\label{thm:intro-dyadic}
Let
\[
  \Pcal=\left\langle a,b\ \middle|\ [a^{2^k},b],\ k=0,1,2,\ldots\right\rangle.
\]
Then $\Pcal$ presents $\Z^2$, and
\[
  \delta_{\Pcal}(n)\fsimeq  n\log n.
\]
In fact, $cn\log n\le \delta_{\Pcal}(n)\le Cn\log n$ for all
sufficiently large $n$ and some constants $0<c<C$.
\end{thm}

\noindent
This statement is restated as \cref{thm:main} in
\cref{sec:dyadic-presentation}.  Its proof is geometric and arithmetic.  The
upper bound decomposes lattice fillings into horizontal dyadic strips, while the
lower bound uses a row-sum argument together with a signed-binary-weight
estimate.

\begin{rem}[A contrasting square-commutator presentation]\label{rem:intro-square-commutator}
The dyadic-strip presentation is deliberately sparse: it permits large
horizontal strips, but only of height one and only in dyadic widths.  This
restriction is important.  If, instead, one adds all axis-parallel square
commutator relators
\[
  [a^m,b^m],\qquad m\ge 1,
\]
to the standard generators of $\Z^2$, then we can prove that the resulting
presentation of $\Z^2$ has linear Dehn function.  More precisely, the presentation
\[
  \left\langle a,b\ \middle|\ [a^m,b^m],\ m=1,2,3,\ldots\right\rangle
\]
has area at most $n/2$ for every null word of length $n$, while
$\Area([a,b]^k)=k$ for every $k\ge1$.  We omit the proof for space reasons.
\end{rem}

The remaining results are based on scalar relation invariants.  Let
$G=\langle X\mid R\rangle$ be a finite presentation, let $F(X)$ be the free
group on $X$, and let $N=\ker(F(X)\to G)$.  A homomorphism $\chi:N\to\Z$ that is
invariant under conjugation by elements of $F(X)$ is called a \emph{scalar
relation invariant}.  Equivalently, $\chi$ factors through the relation module
$N/[F(X),N]$.  Relation modules and the Hopf exact sequence are standard tools
in group homology; see Brown \cite{BrownCohomology}.  For such a homomorphism
put
\[
  \Phi_\chi(n)=\max\{|\chi(w)|:w\in N,
  |w|_X\le n\}.
\]

\begin{thm}[Logarithmic presentations from scalar invariants]\label{thm:intro-general-log}
Let
\[
  G=\langle X\mid R\rangle
\]
be a finite presentation.  Let $F(X)$ be the free group on $X$, and let
$N=\ker(F(X)\to G)$.  Suppose that $\chi:N\to\Z$ is a conjugacy-invariant scalar
relation invariant with infinite image and that, for some $D\ge 1$ and $C>0$,
\[
  \Phi_\chi(n)\le Cn^D+C.
\]
Then there exists an infinite set $S$ of cyclically reduced words in $F(X)$ such
that, for
\[
  \Qcal=\langle X\mid S\rangle,
\]
the group presented by $\Qcal$ is isomorphic to $G$ and
\[
  \delta_{\Qcal}(n)\fsimeq \log n.
\]
More concretely, after dividing $\chi$ by the positive generator of its image,
then replacing it by $-\chi$ if necessary, and choosing a cyclically reduced word
$q\in N$ with $\chi(q)=1$, one may take $S$ to contain all cyclically reduced
words $u\in N$ with $\chi(u)=0$, together with the words $q^{2^k}$, $k\ge 0$.
\end{thm}

\noindent
This statement is proved in \cref{sec:general-relation-invariant}; see
\cref{thm:general-log}.  The construction kills all words in $\ker\chi$ and adds
relators $q^{2^k}$, where $\chi(q)=1$.  It is not meant to be effective in
general: without extra hypotheses, the set of cyclically reduced words in
$N\cap\ker\chi$ need not be decidable.  The logarithmic conclusion here, as
throughout the paper, is meant with respect to the fine comparison relation
$\fsimeq$ introduced in \cref{def:fine-comparison}.

The same scalar mechanism gives polynomial-envelope examples.  Here
\emph{matching peaks} means that a lower bound of order $n^\alpha$ holds along
an infinite sequence of input lengths.  The theorem gives a global upper bound
together with these lower-bound peaks, not a pointwise equivalence statement.

\begin{thm}[Polynomial envelopes from scalar invariants]\label{thm:intro-general-polynomial}
Let
\[
  G=\langle X\mid R\rangle
\]
be a finite presentation.  Let $F(X)$ be the free group on $X$, and let
$N=\ker(F(X)\to G)$.  Assume that there is a conjugacy-invariant scalar
relation invariant $\chi:N\to\Z$ with infinite image and that, for some
$D\ge 1$ and $C>0$, the growth function
\[
  \Phi_\chi(n)=\max\{ |\chi(w)|:w\in N,\ |w|_X\le n\}
\]
satisfies $\Phi_\chi(n)\le Cn^D+C$.

Then, for every $0<\alpha<D$, there exists an infinite set $S$ of cyclically
reduced words in $F(X)$ such that, for
\[
  \Qcal=\langle X\mid S\rangle,
\]
the group presented by $\Qcal$ is isomorphic to $G$ and the following conclusions hold.
\begin{enumerate}
\item There is a constant $C'>0$ such that
\[
  \delta_{\Qcal}(n)\le C' n^\alpha+C'
\]
for all $n\ge1$.
\item If, moreover, for this value of $\alpha$ the hard values $H_j$ from
the polynomial construction are $D$-efficiently realized by $\chi$ in the sense
of \cref{def:efficient-realization}, then there exist a constant $c>0$ and an
infinite sequence $n_i\to\infty$ such that
\[
  \delta_{\Qcal}(n_i)\ge c n_i^\alpha.
\]
\end{enumerate}
\end{thm}

\noindent
This statement is proved in \cref{thm:general-polynomial}.  The hard values
are the integers $H_j$ defined in the polynomial construction immediately before
\cref{def:efficient-realization}.
Applications to $\Z^2$, torsion-free $C'(1/6)$ small-cancellation groups, and word-hyperbolic groups are given in \cref{sec:scalar-applications}.  Signed area in $\Z^2$ has quadratic growth and
its hard values are efficiently realized by lattice polyomino boundaries of area
$M$ and length $O(M^{1/2})$, giving exponents $0<\alpha<2$.  Torsion-free $C'(1/6)$ small-cancellation groups, and more generally word-hyperbolic groups with suitable scalar invariants, give the range $0<\alpha<1$.

The existence of scalar relation invariants is controlled by standard homology.
The Hopf exact sequence gives the following useful source of examples.

\begin{rem}[Homological source of scalar invariants]\label{rem:intro-homological-source}
Let
\[
  G=\langle X\mid R\rangle
\]
be a finite presentation, let $F(X)$ be the free group on $X$, and let
$N=\ker(F(X)\to G)$.  By \cref{cor:existence-scalar-invariant}, a
conjugacy-invariant homomorphism $N\to\Z$ with infinite image exists exactly
when $N/[F(X),N]$ has positive torsion-free rank.  In particular, such a
homomorphism exists whenever
\[
  \operatorname{rank}_{\Z}H_2(G;\Z)+|X|-\operatorname{rank}_{\Z}H_1(G;\Z)>0.
\]
In the same notation,
\[
  \operatorname{rank}_{\Z}\bigl(N/[F(X),N]\bigr)
  =\operatorname{rank}_{\Z}H_2(G;\Z)+|X|-\operatorname{rank}_{\Z}H_1(G;\Z).
\]
Thus the hypothesis holds if
\[
  \operatorname{rank}_{\Z}H_2(G;\Z)>0
  \quad\text{or}\quad
  |X|>b_1(G),
\]
where $b_1(G)=\operatorname{rank}_{\Z}H_1(G;\Z)$.  It holds, for example, for
closed orientable surface groups, because their second homology has rank one.
It also holds whenever $|X|>b_1(G)$; in particular it holds for any finite or
perfect group equipped with a nonempty finite generating set, and more generally
for any group whose chosen finite generating set contains a generator which is
redundant at the level of rational abelianization.  Conversely, this scalar
method gives no invariant when
\[
  \operatorname{rank}_{\Z}H_2(G;\Z)=0
  \quad\text{and}\quad
  |X|=b_1(G)
\]
for the chosen generating alphabet.  Word-hyperbolicity is not needed for existence
of such an invariant; in the hyperbolic case it is used to control the growth
function $\Phi_\chi$.
\end{rem}

The preceding theorems give the following concrete families.

The following statement is proved in
\cref{cor:z2-log-from-general,cor:z2-polynomial-from-general}.
\begin{thm}[Free abelian consequences for $\Z^2$]\label{thm:intro-free-abelian}
Let
\[
  A_2=\Z^2=\langle a,b\mid [a,b]\rangle.
\]
Then the following conclusions hold.
\begin{enumerate}
\item There exists an infinite set $S$ of cyclically reduced words in
$F(a,b)$ such that, for
\[
  \Qcal=\langle a,b\mid S\rangle,
\]
the group presented by $\Qcal$ is isomorphic to $A_2$ and
\[
  \delta_{\Qcal}(n)\fsimeq\log n.
\]
\item For every $0<\alpha<2$, there exists such an infinite presentation
$\Qcal_\alpha$ whose Dehn function satisfies
\[
  \delta_{\Qcal_\alpha}(n)\le Cn^\alpha+C
\]
for all $n\ge1$ and has matching $n^\alpha$-order lower-bound peaks.
\end{enumerate}
\end{thm}

\noindent
We also obtain the corresponding scalar-area result for the standard
presentation of $\Z^d$, for every $d\ge2$: after projecting to the first two
standard generators, the conclusions above extend to standard finite generating
sets of higher-rank free abelian groups; see \cref{cor:zd-from-general}.

The following statement is proved in \cref{cor:small-cancellation-from-general}.
\begin{thm}[Small-cancellation consequences]\label{thm:intro-small-cancellation}
Let $X$ be finite, and let $R\subseteq F(X)$ be a nonempty finite symmetrized
$C'(1/6)$ set of nontrivial cyclically reduced words, none of which is a proper
power.  Put
\[
  G=\langle X\mid R\rangle.
\]
Then the following conclusions hold.
\begin{enumerate}
\item The group $G$ has an infinite presentation on the same generating set $X$
with Dehn function $\fsimeq\log n$.
\item For every $0<\alpha<1$, the group $G$ has an infinite presentation on $X$
whose Dehn function satisfies
\[
  \delta(n)\le Cn^\alpha+C
\]
for all $n\ge1$ and has matching $n^\alpha$-order lower-bound peaks.
\end{enumerate}
\end{thm}

\noindent
The proof of \cref{thm:intro-free-abelian} uses signed area on $\Z^2$.  For
this invariant $\Phi_\chi(n)=O(n^2)$, and lattice polyomino boundaries realize
each required hard area $M$ with length $O(M^{1/2})$.  The proof of
\cref{thm:intro-small-cancellation} uses the relation-module scalar invariant
supplied by the asphericity of $C'(1/6)$ presentations with no proper-power
relators.  In that case $\Phi_\chi(n)=O(n)$, because such groups are
word-hyperbolic.  The standard presentations of closed orientable surface
groups of genus $g\ge2$ satisfy these small-cancellation hypotheses after
symmetrization: the surface relator has length $4g$, and every piece has length
at most one.  Combining \cref{rem:intro-homological-source} with the preceding
theorems also yields the word-hyperbolic corollaries in
\cref{cor:hyperbolic-scalar-invariant,cor:hyperbolic-homological-source}.

The last result records a more flexible, deliberately redundant construction.
It shows that, if the generating set is allowed to contain a generator which
represents the identity, then the presentation-dependence of Dehn functions for
infinite relator sets is universal among finitely generated groups.

The following statement is proved in \cref{thm:arbitrary-fg-continuum}.
\begin{thm}[Continuum many fine classes for arbitrary finitely generated groups]\label{thm:intro-arbitrary-fg-continuum}
Let $G$ be any finitely generated group.  Then there exists a finite generating
alphabet $X$ mapping onto $G$ with the following property.  For every
$0<\alpha<1$, there is an infinite presentation
\[
  \Pcal_\alpha=\langle X\mid \Rcal_\alpha\rangle
\]
of $G$ such that
\[
  \delta_{\Pcal_\alpha}(n)\le C_\alpha n^\alpha+C_\alpha
\]
for all $n\ge1$, and such that $\delta_{\Pcal_\alpha}$ has matching
$n^\alpha$-order lower-bound peaks.  Consequently, as $\alpha$ varies in
$(0,1)$, these presentations determine continuum many pairwise distinct fine
equivalence classes of Dehn functions.
\end{thm}

\noindent
The proof is given in \cref{sec:arbitrary-fg}.  We add a redundant generator
$t$ representing the identity of $G$ and use its exponent sum as a scalar
coordinate.  Suitable infinite relator sets in the $t$-direction encode the
signed coin-counting functions from \cref{lem:general-arithmetic}.  The
construction is intentionally redundant, emphasizing the presentation-dependent
nature of these questions, in contrast to the quasi-isometry-invariant Dehn
spectra of Osin--Rybak~\cite{OsinRybak}.

The author is grateful to Denis Osin for suggesting proving \cref{thm:intro-arbitrary-fg-continuum} and for informing the author about the Osin--Rybak~\cite{OsinRybak} results.

\section{Preliminaries}\label{sec:preliminaries}

Let $X$ be a finite set and let $F(X)$ be the free group on $X$.

\begin{conv}[Presentations used in this paper]\label{conv:presentations}
Unless the word ``finite'' is used explicitly, a presentation in this paper has
the form
\[
  \Pcal=\langle X\mid \Rcal\rangle,
\]
where $X$ is finite and $\Rcal$ is an arbitrary, possibly infinite, set of
cyclically reduced words in $F(X)$.  We write
\[
  G(\Pcal)=F(X)/\ncl_{F(X)}(\Rcal),
\]
where $\ncl_{F(X)}(\Rcal)$ denotes the normal closure of $\Rcal$ in $F(X)$.
When no confusion is possible we write simply $G=G(\Pcal)$.
\end{conv}

\begin{defn}[Length, area, and Dehn functions]\label{def:area-dehn}
Let $\Pcal=\langle X\mid \Rcal\rangle$ be a presentation as in
\cref{conv:presentations}, and put $G=G(\Pcal)$.
\begin{enumerate}[label=\textup{(\arabic*)},leftmargin=*]
\item For $w\in F(X)$, let $|w|_X$ denote the length of the freely reduced
representative of $w$ over the alphabet $X^{\pm1}$.  Let
\[
  \|w\|_X=\min\{|v|_X: v\in F(X)\text{ is conjugate to }w\text{ in }F(X)\}
\]
denote the cyclically reduced length of $w$.  Equivalently, $\|w\|_X$ is the
length of a cyclically reduced representative of the conjugacy class of $w$.

\item If $w\in \ncl_{F(X)}(\Rcal)$ is freely reduced, equivalently if
$w=_G1$, the \emph{$\Pcal$-area} of $w$, denoted $\Area_{\Pcal}(w)$, is the
least integer $m$ for which
\[
  w=\prod_{i=1}^m g_i r_i^{\eps_i} g_i^{-1}
\]
in $F(X)$, where $g_i\in F(X)$, $r_i\in\Rcal$, and
$\eps_i\in\{\pm1\}$.  Equivalently, $\Area_{\Pcal}(w)$ is the smallest number
of $2$-cells in a van Kampen diagram over $\Pcal$ with boundary label $w$.
Each defining relator, regardless of its length, has area cost one.

\item The \emph{Dehn function} of $\Pcal$ is
\[
  \delta_{\Pcal}(n)=\max\{\Area_{\Pcal}(w): w\in \ncl_{F(X)}(\Rcal),
  \ |w|_X\le n,\ w\text{ freely reduced}\}.
\]
Equivalently, the maximum is taken over freely reduced words $w\in F(X)$ such
that $w=_G1$ and $|w|_X\le n$.
\end{enumerate}
The definition agrees with the standard Dehn-function definition for finite
presentations, except that $\Rcal$ may be infinite.  We use the finer
comparison relation from \cref{def:fine-comparison} when distinguishing
logarithmic and sublinear growth.  We do not assert that the resulting growth
type is invariant under
changing the infinite presentation; in fact such invariance generally fails.
\end{defn}

\begin{rem}[One-cell fillings with whiskers]\label{rem:one-cell-whisker}
If a cyclically reduced conjugate of a freely reduced null word $u$ is a
defining relator, then $u$ has area at most one.  Equivalently, after writing
$u$ as a conjugate of its cyclically reduced part, one may realize such a
filling by a single relator cell and, if the chosen boundary basepoint is not on
the cyclically reduced part, by a boundary whisker labelled by the conjugating
word.  In particular, if a word first has to be freely reduced and then cyclically
reduced before it is recognized as a defining relator, the original freely
reduced boundary word is still filled by the same one-cell-with-whisker diagram.
We use this elementary observation below without further comment.
\end{rem}

\begin{defn}[Standard comparison]\label{def:standard-comparison}
For monotone nondecreasing functions $f,g:\N\to[0,\infty)$, say that $f$ is
\emph{standardly dominated} by $g$, and write $f\preceq g$, if there is a
constant $C\ge1$ such that
\[
  f(n)\le Cg(Cn+C)+Cn+C
\]
for all $n$.  We say that $f$ and $g$ are \emph{standardly equivalent}, and write
$f\simeq g$, if $f\preceq g$ and $g\preceq f$.
\end{defn}

The additive linear error term in \cref{def:standard-comparison} is natural for
finite presentations, but it is too coarse for the sublinear examples considered
below.  Indeed, under $\simeq$, every function bounded above by a linear
function is equivalent to $n$; in particular, $1\simeq \log n\simeq n^\alpha
\simeq n$ for $0<\alpha<1$.

\begin{defn}[Fine comparison]\label{def:fine-comparison}
For monotone nondecreasing functions $f,g:\N\to[0,\infty)$, say that $f$ is
\emph{finely dominated} by $g$, and write $f\fpreceq g$, if there is a constant
$C\ge1$ such that
\[
  f(n)\le Cg(Cn+C)+C
\]
for all $n$.  We say that $f$ and $g$ are \emph{finely equivalent}, and write
$f\fsimeq g$, if $f\fpreceq g$ and $g\fpreceq f$.
\end{defn}

The relation $\fpreceq$ allows multiplicative distortion of function values and
linear rescaling of the input, but adds no independent linear error term.
Consequently it distinguishes bounded, logarithmic,
sublinear polynomial, linear, and superlinear growth in the elementary cases
needed in this paper.  All logarithmic and sublinear polynomial comparisons in
this paper are meant with respect to $\fpreceq$ and $\fsimeq$, or equivalently by
the explicit two-sided inequalities displayed in the relevant proofs.  For
superlinear functions such as $n\log n$, the standard and fine comparisons give
the same conclusions in the arguments below.

For group elements $x,y$ we use the convention
\[
  [x,y]=xyx^{-1}y^{-1}.
\]
The standard square grid for $\Z^2$ is oriented so that the generator $a$ points
in the positive horizontal direction and the generator $b$ points in the
positive vertical direction.  A closed word in $a^{\pm1},b^{\pm1}$ is viewed as
a closed edge path in this grid.

\section{The dyadic-strip presentation}\label{sec:dyadic-presentation}

Let
\[
 r_k=[a^{2^k},b]=a^{2^k}ba^{-2^k}b^{-1}\qquad (k\in \Nzero)
\]
and put
\[
  \Pcal=\left\langle a,b\ \middle|\ r_k,\ k\in\Nzero\right\rangle.
\]
Let $G=G(\Pcal)$.  Since $r_0=[a,b]$, $\Pcal$ presents the same group as
$\langle a,b\mid [a,b]\rangle$, namely
\[
  \Z^2=\langle a,b\mid [a,b]\rangle.
\]
Every relator $r_k$ is counted as one $2$-cell, independently of $k$.

For a freely reduced word $w\in \F$ such that $w=_G1$, let
$\Area_{\Pcal}(w)$ denote the minimum number of $2$-cells in a van Kampen diagram
over $\Pcal$ with boundary label $w$.  With the notation of \cref{def:area-dehn}, the Dehn
function of $\Pcal$ is
\[
  \delta_{\Pcal}(n)=\max\{\Area_{\Pcal}(w): |w|_{\{a,b\}}\le n,
  w=_G1,\ w\text{ freely reduced}\}.
\]

\begin{thm}\label{thm:main}
For the presentation
\[
  \Pcal=\left\langle a,b\ \middle|\ [a^{2^k},b],\ k=0,1,2,\ldots\right\rangle
\]
of $\Z^2$, there exist constants $0<c<C$ and $n_0\ge 1$ such that for all
$n\ge n_0$,
\[
  c n\log n\le \delta_{\Pcal}(n)\le C n\log n.
\]
In particular, $\delta_{\Pcal}$ is superlinear and subquadratic.
\end{thm}

\begin{rem}
This theorem is a statement about the Dehn function of the specified infinite
presentation.  It is not in conflict with the usual finite-presentation theory,
where subquadratic Dehn function forces linear Dehn function.  The relators here
have unbounded length, and long relators are counted with area one.
\end{rem}

\subsection{The geometric model}

We use the standard square tiling of $\R^2$, with vertices $\Z^2$.  The letter
$a$ is interpreted as the unit horizontal step $(x,y)\mapsto (x+1,y)$, and the
letter $b$ as the unit vertical step $(x,y)\mapsto (x,y+1)$.  Thus every word
$w\in\F$ gives a lattice path in this grid, starting at the origin.

The relator
\[
 r_k=a^{2^k}ba^{-2^k}b^{-1}
\]
is the boundary of a horizontal rectangle of width $2^k$ and height $1$.
More precisely, a translate of an $r_k$-cell with lower-left corner $(x,y)$
contributes the cellular $2$-chain
\[
  \sum_{i=0}^{2^k-1} s_{x+i,y},
\]
where $s_{u,v}$ denotes the positively oriented unit square
$[u,u+1]\times[v,v+1]$.  An oppositely oriented $r_k$-cell contributes the
negative of this chain.

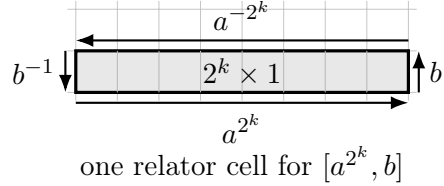
\begin{figure}[htbp]
\centering
\begin{tikzpicture}[scale=.55,>=Latex]
  \draw[step=1,gray!45,very thin] (-.2,-.2) grid (8.2,2.2);
  \fill[gray!18] (0,0) rectangle (8,1);
  \draw[very thick] (0,0) rectangle (8,1);
  \foreach \x in {1,...,7} {\draw[gray!55] (\x,0) -- (\x,1);}
  \draw[->,thick] (0,-.25) -- (8,-.25) node[midway,below] {$a^{2^k}$};
  \draw[->,thick] (8.25,0) -- (8.25,1) node[midway,right] {$b$};
  \draw[->,thick] (8,1.25) -- (0,1.25) node[midway,above] {$a^{-2^k}$};
  \draw[->,thick] (-.25,1) -- (-.25,0) node[midway,left] {$b^{-1}$};
  \node at (4,.5) {$2^k\times 1$};
  \node[below] at (4,-1.05) {one relator cell for $[a^{2^k},b]$};
\end{tikzpicture}
\caption{A relator $[a^{2^k},b]$ fills one horizontal dyadic strip of width $2^k$ and height $1$.}
\label{fig:dyadic-strip-cell}
\end{figure}

\begin{lem}[Projected $2$-chains]\label{lem:projected-chain}
Let $\mathcal D$ be a van Kampen diagram over $\Pcal$ with boundary label $w$.  Then $\mathcal D$
determines a finite cellular $2$-chain $C_{\mathcal D}$ in the standard square tiling of
$\R^2$ such that
\[
  \partial C_{\mathcal D}=\gamma_w,
\]
where $\gamma_w$ is the cellular $1$-chain traced by the lattice path labelled
by $w$.  Each $2$-cell of type $r_k^{\pm1}$ contributes, with sign, a translate
of a horizontal $2^k\times 1$ rectangle.
\end{lem}

\begin{proof}
Choose a base vertex of $\mathcal D$ mapping to $(0,0)\in\Z^2$.  To every vertex $v$ of
$\mathcal D$ assign the exponent-sum vector of the label of any path from the base vertex
to $v$.  This definition is well-defined because the boundary word of each relator $r_k$
has total exponent sum zero in both $a$ and $b$, and $\mathcal D$ is simply connected.

Now consider a $2$-cell $\Pi$ of $\mathcal D$.  Its boundary label is a cyclic conjugate
of $r_k$ or $r_k^{-1}$ for some $k$.  The chosen starting point on the boundary
of the cell only translates the resulting rectangle: under the vertex map just
defined, the boundary of $\Pi$ is sent to the boundary of a translate of the
horizontal rectangle $[0,2^k]\times[0,1]$, with one of the two orientations.
Replace $\Pi$ by the corresponding signed sum of $2^k$ unit squares.  Summing
over all $2$-cells of $\mathcal D$ gives a finite cellular $2$-chain $C_{\mathcal D}$.

The boundary of this sum is the sum of the projected boundaries of the $2$-cells.
Projected interior edges of $\mathcal D$ occur twice with opposite orientations and cancel.
The only remaining boundary is the projected outer boundary, namely
$\gamma_w$.  Hence $\partial C_{\mathcal D}=\gamma_w$.
\end{proof}

\begin{lem}[No finite $2$-cycles in the grid]\label{lem:no-cycles}
Let
\[
 C=\sum_{(x,y)\in\Z^2} c_{x,y}s_{x,y}
\]
be a finite cellular $2$-chain in the standard square tiling of $\R^2$, with
$c_{x,y}\in\Z$.  If $\partial C=0$, then $C=0$.
\end{lem}

\begin{proof}
Suppose $C\ne0$.  Choose $x_0$ maximal such that $c_{x_0,y}\ne0$ for some
$y\in\Z$, and choose such a $y_0$.  The right vertical edge of the square
$s_{x_0,y_0}$ occurs in $\partial C$ with coefficient $c_{x_0,y_0}$, since no
square immediately to its right has nonzero coefficient.  Thus $\partial C\ne0$,
a contradiction.  Hence $C=0$.
\end{proof}

\begin{cor}\label{cor:square-chain}
Let $Q_L$ be the positively oriented $L\times L$ square cellular $2$-chain
\[
  Q_L=\sum_{x=0}^{L-1}\sum_{y=0}^{L-1}s_{x,y}.
\]
If $\mathcal D$ is a van Kampen diagram over $\Pcal$ whose boundary path is the positively
oriented boundary of $Q_L$, then $C_{\mathcal D}=Q_L$.
\end{cor}

\begin{proof}
By \cref{lem:projected-chain}, $\partial C_{\mathcal D}=\partial Q_L$.  Therefore
$C_{\mathcal D}-Q_L$ is a finite cellular $2$-cycle in the grid.  By \cref{lem:no-cycles},
$C_{\mathcal D}-Q_L=0$.
\end{proof}

\subsection{The upper bound}\label{sec:upper-bound}

We first fill simple lattice loops by slicing their interiors into horizontal
unit-height strips.  The schematic picture in \cref{fig:row-decomposition}
shows the operation used in the proof: each horizontal row is a disjoint union
of integer intervals, and each such interval is decomposed into dyadic pieces.

\begin{figure}[htbp]
\centering
\begin{tikzpicture}[scale=.45]
  \draw[step=1,gray!45,very thin] (-.2,-.2) grid (14.2,5.2);
  \fill[gray!15] (1,0) rectangle (10,1);
  \fill[gray!15] (0,1) rectangle (13,2);
  \fill[gray!15] (2,2) rectangle (12,3);
  \fill[gray!15] (3,3) rectangle (9,4);
  \fill[gray!15] (5,4) rectangle (8,5);
  \draw[very thick] (1,0) -- (10,0) -- (10,1) -- (13,1) -- (13,2) -- (12,2) -- (12,3) -- (9,3) -- (9,4) -- (8,4) -- (8,5) -- (5,5) -- (5,4) -- (3,4) -- (3,3) -- (2,3) -- (2,2) -- (0,2) -- (0,1) -- (1,1) -- cycle;
  \draw[very thick] (0,1) -- (8,1);
  \draw[very thick,dashed] (8,1) -- (12,1);
  \draw[very thick,dotted] (12,1) -- (13,1);
  \draw[|-|] (0,-.65) -- (8,-.65) node[midway,below] {$8$};
  \draw[|-|] (8,-.65) -- (12,-.65) node[midway,below] {$4$};
  \draw[|-|] (12,-.65) -- (13,-.65) node[midway,below] {$1$};
  \node[right] at (13.3,1.5) {$13=8+4+1$};
\end{tikzpicture}
\caption{Horizontal slicing of a lattice region.  A row interval of length $13$ is filled by dyadic strips of lengths $8$, $4$, and $1$.}
\label{fig:row-decomposition}
\end{figure}
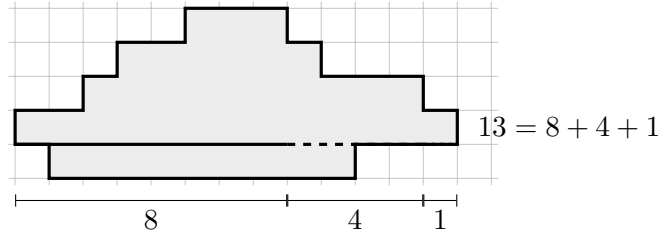

\begin{lem}[Filling a simple loop]\label{lem:simple-loop-upper}
Let $\gamma$ be a simple closed lattice path of length $p\ge4$.  Then $\gamma$
has a filling over $\Pcal$ with at most
\[
  \frac p2\bigl(\lfloor \log_2 p\rfloor+1\bigr)
\]
$2$-cells.
\end{lem}

\begin{proof}
Let $\Omega$ be the compact region bounded by $\gamma$.  Since $\gamma$ is a
simple closed lattice path, $\Omega$ is a simply connected lattice polygon.
For each integer $y$, consider the open horizontal strip
\[
  S_y^\circ=\R\times(y,y+1).
\]
The intersection $\operatorname{int}(\Omega)\cap S_y^\circ$ is a finite disjoint union of open
rectangles
\[
  (u,v)\times(y,y+1),\qquad u,v\in\Z,
\]
with positive integer lengths $\ell=v-u$.  We fill the closures of these
rectangles.

For a fixed strip $S_y^\circ$, the endpoints of these intervals occur exactly at the
vertical boundary edges of $\gamma$ crossing that strip.  As one moves across the
strip from left to right, these vertical boundary edges alternate between entry
and exit edges.  Therefore the number of interval components in
$\operatorname{int}(\Omega)\cap S_y^\circ$ is half the number of vertical boundary edges of
$\gamma$ in that strip.  Summing over all $y$, the total number of interval
components is half the total number of vertical edges of $\gamma$, and hence is
at most $p/2$.

Now take one interval component $(u,v)\times(y,y+1)$ of length $\ell=v-u$.
Write $\ell$ in binary,
\[
  \ell=2^{k_1}+\cdots+2^{k_m},
\]
where the $2^{k_i}$ are the distinct powers of two appearing in the binary
expansion of $\ell$.  Then
\[
  m\le \lfloor \log_2 \ell\rfloor+1\le \lfloor \log_2 p\rfloor+1.
\]
Partition the closed interval $[u,v]$ consecutively into subintervals of lengths
$2^{k_1},\ldots,2^{k_m}$.  Each closed subrectangle of size $2^{k_i}\times1$ is filled
by one translate of the relator cell $r_{k_i}$.

Carrying out this construction for every horizontal interval component tiles $\Omega$ by horizontal
dyadic rectangles.  Since $\Omega$ is a topological disk, the resulting tiling gives a van
Kampen diagram with boundary label $\gamma$ after subdividing horizontal sides
into unit $a$-edges: interior horizontal and vertical subdivision edges are glued
in inverse-oriented pairs, while the remaining boundary is the original lattice
loop.  We use inverse relator cells if the orientation of $\gamma$ is opposite to
the positive boundary orientation of $\Omega$.  The number of cells is at most
\[
  \frac p2\bigl(\lfloor \log_2 p\rfloor+1\bigr),
\]
as claimed.
\end{proof}

\begin{lem}[Cutting a closed walk into simple loops]\label{lem:cutting}
Let $\gamma$ be a closed lattice walk of length $n$.  Then, for purposes of
filling, $\gamma$ can be cut into finitely many simple closed lattice paths
$\gamma_1,\ldots,\gamma_s$ whose lengths $p_i=|\gamma_i|$ satisfy
\[
  \sum_{i=1}^s p_i\le n.
\]
More precisely, if each $\gamma_i$ has a filling of area $A_i$, then $\gamma$ has
a filling of area at most $\sum_i A_i$.
\end{lem}

\begin{proof}
We argue by induction on $n$.  If the closed walk has no repeated vertices except
for the equality of its initial and terminal vertices, then it is simple and there
is nothing to prove.

Otherwise, choose two distinct times at which $\gamma$ visits the same vertex.
The segment of $\gamma$ between those two visits is a closed lattice walk
$\beta$, and the complementary part is another closed lattice walk $\alpha$.
Their lengths add to $n$.  A filling of $\beta$ can be attached at the repeated
vertex to a filling of $\alpha$, producing a filling of $\gamma$, with areas
adding.  Applying the induction hypothesis to $\alpha$ and $\beta$, and discarding
any trivial length-zero loops that arise, gives the required simple loops.  The
sum of their lengths is at most $n$.
\end{proof}

\begin{prop}[Upper bound]\label{prop:upper}
For every $n\ge2$,
\[
  \delta_{\Pcal}(n)\le \frac n2\bigl(\lfloor\log_2 n\rfloor+1\bigr).
\]
In particular, $\delta_{\Pcal}(n)=O(n\log n)$.
\end{prop}

\begin{proof}
Let $w$ be a freely reduced word over $\{a^{\pm1},b^{\pm1}\}$ with $w=_G1$ and $|w|_{\{a,b\}}\le n$.  The word
$w$ traces a closed lattice walk of length at most $n$.  By \cref{lem:cutting},
it suffices to fill simple closed lattice paths $\gamma_1,\ldots,\gamma_s$ with
lengths $p_1,\ldots,p_s$ satisfying $\sum_i p_i\le n$.

By \cref{lem:simple-loop-upper}, the total area is at most
\[
  \sum_{i=1}^s \frac{p_i}{2}\bigl(\lfloor\log_2 p_i\rfloor+1\bigr)
  \le
  \sum_{i=1}^s \frac{p_i}{2}\bigl(\lfloor\log_2 n\rfloor+1\bigr)
  \le
  \frac n2\bigl(\lfloor\log_2 n\rfloor+1\bigr).
\]
Thus $\Area_{\Pcal}(w)$ is at most the displayed quantity.  Taking the maximum
over all freely reduced words $w$ with $w=_G1$ and $|w|_{\{a,b\}}\le n$ gives
the result.
\end{proof}

\subsection{Signed binary weight}

The lower bound uses a small amount of elementary arithmetic.

\begin{defn}\label{def:signed-weight}
For an integer $n\ge0$, the \emph{signed binary weight} $\tau(n)$ is the least integer $r\ge0$ such
that
\[
  n=\sum_{i=1}^r \eps_i2^{k_i},
  \qquad \eps_i\in\{\pm1\},\quad k_i\in\Nzero.
\]
Thus $\tau(n)$ is the minimal number of signed powers of two needed to express
$n$.  We put $\tau(0)=0$.
\end{defn}

\begin{lem}[Recurrences for signed binary weight]\label{lem:tau-recurrence}
For every $m\ge0$,
\[
  \tau(2m)=\tau(m)
\]
and
\[
  \tau(2m+1)=1+\min\{\tau(m),\tau(m+1)\}.
\]
\end{lem}

\begin{proof}
The inequality $\tau(2m)\le\tau(m)$ follows by doubling a shortest signed binary
expression for $m$.
Conversely, take a shortest expression for $2m$.  Since $2m$ is even, the number
of terms equal to $\pm1$ is even.  Pair these $\pm1$ terms arbitrarily.  Each
pair is either $(1,1)$, $(-1,-1)$, or $(1,-1)$, and may be replaced respectively
by $2$, $-2$, or $0$, without increasing the number of terms.  Hence $2m$ has an
expression of length at most $\tau(2m)$ in which every term is divisible by $2$.
Dividing by $2$ gives an expression for $m$ of length at most $\tau(2m)$.  Thus
$\tau(m)\le\tau(2m)$, proving $\tau(2m)=\tau(m)$.

For the odd case, a shortest expression for $2m+1$ has an odd number of terms
$\pm1$.  Pair all but one of these terms as above.  This pairing produces, without
increasing length, an expression with exactly one term equal to $1$ or exactly
one term equal to $-1$, and all other terms divisible by $2$.
If the remaining odd term is $1$, then after subtracting it and dividing by $2$
we get an expression for $m$.  If it is $-1$, then after adding $1$ and dividing
by $2$ we get an expression for $m+1$.  Therefore
\[
  \tau(2m+1)\ge 1+\min\{\tau(m),\tau(m+1)\}.
\]
The reverse inequality follows by taking a shortest expression for either $m$ or
$m+1$, doubling it, and then adding $1$ in the first case or subtracting $1$ in
the second case.  Hence the recurrence follows.
\end{proof}

\begin{lem}\label{lem:Lt-weight}
For
\[
  L_t=1+4+4^2+\cdots+4^{t-1}=\frac{4^t-1}{3}\qquad (t\ge0),
\]
one has
\[
  \tau(L_t)=t.
\]
\end{lem}

\begin{proof}
Let
\[
  M_t=2L_t+1.
\]
We prove simultaneously, by induction on $t\ge0$, that
\[
  \tau(L_t)=t\qquad\text{and}\qquad \tau(M_t)=t+1.
\]
For $t=0$, we have $L_0=0$ and $M_0=1$, so the claim is immediate.

Assume the result is known for $t-1$, with $t\ge1$.  Since
\[
  L_t=4L_{t-1}+1=2(2L_{t-1})+1,
\]
\cref{lem:tau-recurrence} gives
\[
\begin{aligned}
  \tau(L_t)
   &=1+\min\{\tau(2L_{t-1}),\tau(2L_{t-1}+1)\}  \\
   &=1+\min\{\tau(L_{t-1}),\tau(M_{t-1})\}  \\
   &=1+\min\{t-1,t\}=t.
\end{aligned}
\]
Next,
\[
  M_t=2L_t+1.
\]
Again using \cref{lem:tau-recurrence},
\[
  \tau(M_t)=1+\min\{\tau(L_t),\tau(L_t+1)\}.
\]
But
\[
  L_t+1=2M_{t-1}.
\]
Therefore, by \cref{lem:tau-recurrence} and the induction hypothesis,
\[
  \tau(L_t+1)=\tau(2M_{t-1})=\tau(M_{t-1})=t.
\]
Together with $\tau(L_t)=t$, these inequalities give
\[
  \tau(M_t)=1+\min\{t,t\}=t+1.
\]
The induction is complete.
\end{proof}

\subsection{The lower bound}\label{sec:lower-bound}

For an integer $L\ge1$, let
\[
  w_L=[a^L,b^L]=a^Lb^La^{-L}b^{-L}.
\]
The word $w_L$ labels the positively oriented boundary of the $L\times L$ square
$Q_L$.

\begin{lem}[A row-sum lower bound]\label{lem:row-lower}
For every $L\ge1$,
\[
  \Area_{\Pcal}(w_L)\ge L\tau(L).
\]
\end{lem}

\begin{proof}
Let $\mathcal D$ be an arbitrary van Kampen diagram over $\Pcal$ with boundary label
$w_L$.  By \cref{cor:square-chain}, its projected $2$-chain satisfies
\[
  C_{\mathcal D}=Q_L.
\]

For each $y\in\Z$, define the row-sum homomorphism on finite cellular
$2$-chains by
\[
  \rho_y\left(\sum_{(x,z)\in\Z^2} c_{x,z}s_{x,z}\right)=\sum_{x\in\Z}c_{x,y}.
\]
For $y=0,1,\ldots,L-1$, the $y$-th row of $Q_L$ consists of exactly $L$ unit
squares with coefficient $1$.  Hence
\[
  \rho_y(C_{\mathcal D})=\rho_y(Q_L)=L.
\]

Let $\nu_y$ be the number of $2$-cells of $\mathcal D$ whose projected rectangle lies in the
horizontal strip $\R\times[y,y+1]$.  Each projected $r_k^{\pm1}$-cell is a
horizontal $2^k\times1$ rectangle and hence lies in exactly one such unit-height
strip.  Thus summing the quantities $\nu_y$ over rows counts each cell at most
once.  A cell of type $r_k^{\pm1}$ in this strip contributes $\pm 2^k$ to
$\rho_y(C_{\mathcal D})$, and cells in other strips contribute $0$.  Therefore
$L$ is expressible as a signed sum of $\nu_y$ powers of two.  By
the definition of $\tau(L)$,
\[
  \nu_y\ge \tau(L).
\]
Summing over $y=0,1,\ldots,L-1$, we obtain
\[
  \#\{\text{$2$-cells of }\mathcal D\}\ge \sum_{y=0}^{L-1} \nu_y\ge L\tau(L).
\]
Since $\mathcal D$ was arbitrary, $\Area_{\Pcal}(w_L)\ge L\tau(L)$.
\end{proof}

\begin{prop}[Lower bound]\label{prop:lower}
There exists $c>0$ and $n_0\ge1$ such that for every $n\ge n_0$,
\[
  \delta_{\Pcal}(n)\ge c n\log n.
\]
\end{prop}

\begin{proof}
Let
\[
  L_t=\frac{4^t-1}{3}\qquad (t\ge1)
\]
and consider
\[
  w_t=w_{L_t}=[a^{L_t},b^{L_t}].
\]
Then
\[
  |w_t|=4L_t.
\]
By \cref{lem:row-lower,lem:Lt-weight},
\[
  \Area_{\Pcal}(w_t)\ge L_t\tau(L_t)=L_t t.
\]
Thus, putting $n_t=4L_t$, we have
\[
  \delta_{\Pcal}(n_t)\ge L_t t=\frac{n_t}{4}t.
\]
Since
\[
  n_t=4\frac{4^t-1}{3},
\]
we have $t$ comparable to $\log n_t$.  In particular, for all sufficiently large
$t$,
\[
  \delta_{\Pcal}(n_t)\ge c_1 n_t\log n_t
\]
for some constant $c_1>0$.

It remains to pass from the subsequence $n_t$ to all large $n$.  The sequence
$n_t$ has bounded multiplicative gaps, since
\[
  n_{t+1}=4n_t+4\le5n_t\qquad (t\ge1).
\]
Given large $n$, choose $t$ such that
\[
  n_t\le n<n_{t+1}.
\]
Then $n_t\ge n/5$.  Since $\delta_{\Pcal}$ is nondecreasing,
\[
  \delta_{\Pcal}(n)\ge \delta_{\Pcal}(n_t)
  \ge c_1 n_t\log n_t
  \ge c n\log n
\]
for some constant $c>0$ and all sufficiently large $n$.
\end{proof}

\subsection{Completing the proof of the dyadic-strip theorem}\label{sec:proof-main}

\begin{proof}[Proof of \cref{thm:main}]
The presentation $\Pcal$ defines $\Z^2$ because its first relator is
$r_0=[a,b]$, and all other relators already hold in $\Z^2$.

The upper bound
\[
  \delta_{\Pcal}(n)=O(n\log n)
\]
is \cref{prop:upper}.  The lower bound
\[
  \delta_{\Pcal}(n)=\Omega(n\log n)
\]
is \cref{prop:lower}.  Hence $\delta_{\Pcal}(n)$ has order $n\log n$.
Consequently,
\[
  \frac{\delta_{\Pcal}(n)}{n}\to\infty
  \qquad\text{and}\qquad
  \frac{\delta_{\Pcal}(n)}{n^2}\to0.
\]
Equivalently, the Dehn function of this infinite presentation is superlinear but
subquadratic.
\end{proof}

\section{A general scalar relation-invariant construction}\label{sec:general-relation-invariant}

The dyadic-strip construction is geometric.  This section gives a more algebraic mechanism producing logarithmic and sublinear Dehn functions for specified infinite presentations, starting from a conjugacy-invariant homomorphism on the relation subgroup.  The estimates are stated in terms of the growth of that invariant on relation words.

Let
\[
  G=\langle X\mid R\rangle
\]
be a fixed finite presentation, where $X$ is finite, and let
\[
  N=\ker(F(X)\to G).
\]
\begin{defn}[Scalar relation invariant]\label{def:scalar-relation-invariant}
A homomorphism
\[
  \chi:N\to\Z
\]
is a \emph{conjugacy-invariant relation invariant} if
\[
  \chi(fuf^{-1})=\chi(u)
\]
for all $f\in F(X)$ and $u\in N$.  Equivalently, $\chi$ factors through the
quotient $N/[F(X),N]$.
\end{defn}

Assume throughout this section, unless explicitly stated otherwise, that
$\chi(N)=\Z$.  Choose $q\in N$ with $\chi(q)=1$.  Replacing $q$ by a cyclically
reduced conjugate if necessary, we assume that $q$ is cyclically reduced.  This
replacement does not change $\chi(q)$, by conjugacy-invariance.

Define the \emph{growth function of $\chi$} by
\[
  \Phi_\chi(n)=\max\{|\chi(w)|: w\in N,\ |w|_X\le n\}.
\]
The maximum is finite for every $n$, since there are only finitely many words of
length at most $n$ over $X^{\pm1}$.  The special case used below for small-cancellation groups and word-hyperbolic groups is the case
$\Phi_\chi(n)=O(n)$.  The signed-area invariant for $\Z^2$ instead satisfies
$\Phi_\chi(n)=O(n^2)$.

\subsection{The logarithmic construction}

Let $S_{\log}$ be the set consisting of all nontrivial cyclically reduced words
$u\in N$ with $\chi(u)=0$, together with the cyclically reduced words
$q^{2^k}$, $k=0,1,2,\ldots$.  Define the infinite presentation
\[
  \Qcal_{\log}=\langle X\mid S_{\log}\rangle.
\]

\begin{rem}\label{rem:general-log-noneffective}
The definition of $S_{\log}$ is generally non-effective: it uses all cyclically reduced words in $N\cap\ker\chi$, and no decidability or recursive enumerability assertion is made without extra hypotheses.
\end{rem}

\begin{prop}\label{prop:general-log-presents-G}
The presentation $\Qcal_{\log}$ presents the group $G$.
\end{prop}

\begin{proof}
Every defining relator of $\Qcal_{\log}$ belongs to $N$, so the identity map on
$X$ induces a surjective homomorphism from the group presented by
$\Qcal_{\log}$ onto $G$.  Conversely, let $w\in N$.  Put $A=\chi(w)$.  Then
\[
  \chi(wq^{-A})=0.
\]
If the freely reduced word representing $wq^{-A}$ is nontrivial, any cyclically
reduced conjugate of it is one of the defining relators of $\Qcal_{\log}$; if it
is trivial, no relator is needed.  In either case $wq^{-A}$ belongs to the normal
closure of the defining relators of $\Qcal_{\log}$.  Also $q=q^{2^0}$ is one of
the defining relators.  Hence $q^A$ belongs to the same normal closure, and
therefore so does
\[
  w=(wq^{-A})q^A.
\]
Thus the normal closure of the defining relators of $\Qcal_{\log}$ is exactly
$N$, and the presented group is $F(X)/N\cong G$.
\end{proof}

\begin{thm}[Logarithmic presentations from polynomially bounded invariants]\label{thm:general-log}
Suppose that there are constants $C_0,D>0$ such that
\[
  \Phi_\chi(n)\le C_0n^D+C_0
\]
for all $n\ge1$.  Then
\[
  \delta_{\Qcal_{\log}}(n)\fsimeq \log n.
\]
\end{thm}

\begin{proof}
Let $w\in N$ be represented by a word of length at most $n$, and put
$A=\chi(w)$.  A cyclically reduced conjugate of $wq^{-A}$ is either empty or a defining
relator of $\Qcal_{\log}$.  By \cref{rem:one-cell-whisker}, this contributes
at most one cell and reduces the filling problem for $w$ to the filling problem
for $q^A$.
Write $A$ as a signed binary sum
\[
  A=\sum_{i=1}^r\eps_i2^{k_i},\qquad \eps_i\in\{\pm1\}.
\]
Using the usual binary expansion, one can choose such an expression with
\[
  r\le \lfloor \log_2 |A|\rfloor+1
\]
when $A\ne0$, and with $r=0$ when $A=0$.  Each factor
$q^{\eps_i2^{k_i}}$ is filled by one relator cell, with the opposite orientation
if $\eps_i=-1$.  Since $|A|\le \Phi_\chi(n)\le C_0n^D+C_0$, this estimate gives
\[
  \Area_{\Qcal_{\log}}(w)
  \le 1+\log_2(C_0n^D+C_0)+O(1)
  \le C_1\log n+C_1.
\]
Hence
\[
  \delta_{\Qcal_{\log}}(n)\fpreceq \log n.
\]

For the lower bound let
\[
  M_t=1+4+4^2+\cdots+4^{t-1}=\frac{4^t-1}{3}.
\]
As proved in \cref{lem:Lt-weight}, the least number of signed powers of two
needed to express $M_t$ is exactly $t$.  Consider the word
\[
  w_t=q^{M_t}.
\]
Because $q$ is cyclically reduced,
\[
  |w_t|=M_t|q|.
\]
Also $\chi(w_t)=M_t$.  If $\mathcal D_t$ is a van Kampen diagram over $\Qcal_{\log}$ for
$w_t$ with $\nu_t$ two-cells, then the usual boundary relation in $F(X)$ has the
form
\[
  w_t=\prod_{i=1}^{\nu_t} g_i r_i^{\varepsilon_i} g_i^{-1},
  \qquad g_i\in F(X),\quad \varepsilon_i\in\{\pm1\},\quad r_i\in S_{\log},
\]
up to free reduction.  Applying $\chi$ to this equality, and using
conjugacy-invariance, gives an expression for $M_t$ as a signed sum of at most
$\nu_t$ powers of two.  Here cyclic conjugates of relators have the same $\chi$-value because
$\chi$ is conjugacy-invariant; cells with boundary label in $\ker\chi$
contribute $0$, while cells of type $q^{2^k}$ contribute $\pm 2^k$.
Therefore
\[
  \nu_t\ge t.
\]
Thus
\[
  \Area_{\Qcal_{\log}}(w_t)\ge t\simeq \log M_t\simeq \log |w_t|.
\]
Since $|w_{t+1}|/|w_t|$ is bounded above independently of $t$, monotonicity of
the Dehn function upgrades this subsequence lower bound to a lower bound for all
large $n$.  Indeed, choose $t$ with $|w_t|\le n<|w_{t+1}|$.  Then
$|w_t|\ge C_2^{-1}n$ for a constant $C_2$, and therefore
\[
  \delta_{\Qcal_{\log}}(n)\ge \delta_{\Qcal_{\log}}(|w_t|)
  \ge c_2\log |w_t|-C_3
  \ge c_3\log n-C_4.
\]
Thus
\[
\log n \fpreceq \delta_{\Qcal_{\log}}(n).
\]

Combining the upper and lower bounds proves the theorem.
\end{proof}

\subsection{Polynomial envelope constructions}

The polynomial-envelope construction also works for a general relation invariant.  The upper bound is governed by $\Phi_\chi$; matching lower peaks additionally require large hard values of $\chi$ to be realized by sufficiently short relation words.

Fix $D>0$ and assume that
\[
  \Phi_\chi(n)\le C_0n^D+C_0
\]
for all $n\ge1$.  Let $0<\alpha<D$ and put
\[
  \beta=\frac{\alpha}{D}\in(0,1),
  \qquad
  p_\beta=\frac1{1-\beta}.
\]
Choose a divisibility sequence
\[
  1=s_0\mid s_1\mid s_2\mid\cdots
\]
as follows.  Having chosen $s_j$, set
\[
  m_j=2\left\lceil\frac{s_j^{p_\beta-1}}2\right\rceil,
  \qquad s_{j+1}=m_js_j.
\]
Then $m_j$ is even and
\[
  s_j^{p_\beta}\le s_{j+1}\le 3s_j^{p_\beta}.
\]
Let $S_\beta=\{s_j:j\ge0\}$ and define
\[
  \tau_\beta(n)=\min\left\{r:n=\sum_{i=1}^r\eps_i s_{j_i},\
  \ \eps_i\in\{\pm1\},\ j_i\ge0\right\}
\]
for $n\in\Z$, with $\tau_\beta(0)=0$.

\begin{lem}\label{lem:general-arithmetic}
There are constants $C,c>0$ such that
\[
  \tau_\beta(n)\le C|n|^\beta+C
\]
for every $n\in\Z$.  Moreover, for
\[
  H_j=\frac{s_{j+1}}2,
\]
one has
\[
  \tau_\beta(H_j)\ge cH_j^\beta
\]
for every $j$.
\end{lem}

\begin{proof}
The estimate is elementary; we recall the proof.  Assume $n\ge1$ and choose $J$ so that
$s_J\le n<s_{J+1}$.  Using the divisibility chain, write $n$ in mixed-radix
form
\[
  n=d_Js_J+d_{J-1}s_{J-1}+\cdots+d_0s_0,
\]
where $d_J=\lfloor n/s_J\rfloor$ and $0\le d_i<m_i=s_{i+1}/s_i$ for
$0\le i<J$.  Therefore
\[
  \tau_\beta(n)\le \frac{n}{s_J}+\sum_{i=0}^{J-1}m_i.
\]
Since $n<s_{J+1}\le 3s_J^{p_\beta}$,
\[
  \frac{n}{s_J}\le 3^{1/p_\beta}n^{1-1/p_\beta}=3^{1/p_\beta}n^\beta.
\]
The sequence $m_i$ is eventually at least geometrically increasing.  Indeed,
from $s_i^{p_\beta}\le s_{i+1}\le3s_i^{p_\beta}$ we get
\[
  s_i^{p_\beta-1}\le m_i\le 3s_i^{p_\beta-1},
\]
and hence $m_{i+1}/m_i\to\infty$.  In particular, after discarding finitely
many initial indices, the ratios $m_{i+1}/m_i$ are at least $2$, so each partial
sum is bounded by a fixed multiple of its last term.  After changing the
constant to absorb the finite initial part, we have
\[
  \sum_{i=0}^{J-1}m_i\le C_1m_{J-1}.
\]
Furthermore
\[
  m_{J-1}\le 3s_{J-1}^{p_\beta-1}
  \le 3s_J^{(p_\beta-1)/p_\beta}
  \le 3n^\beta.
\]
This estimate proves the upper bound for $n\ge1$; the case $n\le0$ follows
by symmetry.

For the lower bound, suppose that
\[
  H_j=\sum_{i=1}^r\eps_i s_{k_i}
\]
is a signed $S_\beta$-expansion.  Reduce modulo $s_{j+1}$.  All summands with
$k_i\ge j+1$ vanish modulo $s_{j+1}$, and all remaining summands have absolute
value at most $s_j$.  Since $H_j=s_{j+1}/2$ has distance $s_{j+1}/2$ from the
nearest multiple of $s_{j+1}$, one must have
\[
  rs_j\ge \frac{s_{j+1}}2.
\]
Thus
\[
  r\ge \frac{s_{j+1}}{2s_j}=\frac{m_j}{2}.
\]
The growth inequalities give $m_j\ge s_j^{p_\beta-1}$ and
$H_j=s_{j+1}/2\le (3/2)s_j^{p_\beta}$.  Since
$p_\beta\beta=p_\beta-1$, the last lower bound is bounded below by a positive
constant times $H_j^\beta$.  Taking the minimum over all signed expansions
proves the lower bound.
\end{proof}

Let $S_\alpha$ be the set consisting of all nontrivial cyclically reduced words
$u\in N$ with $\chi(u)=0$, together with the cyclically reduced words
$q^{s_j}$, $j=0,1,2,\ldots$.  Define
\[
  \Qcal_\alpha=\langle X\mid S_\alpha\rangle.
\]
As before, $s_0=1$, so $q$ is one of the defining relators, and the same proof
as in \cref{prop:general-log-presents-G} shows that $\Qcal_\alpha$ presents
$G$.  As in \cref{rem:general-log-noneffective}, this definition is not
asserted to be effective in general.

\begin{defn}\label{def:efficient-realization}
With notation as above, say that the \emph{hard values} $H_j=s_{j+1}/2$ are
\emph{$D$-efficiently realized by $\chi$} if there exist a constant $K>0$ and
words $z_j\in N$ such that
\[
  \chi(z_j)=H_j,
  \qquad
  |z_j|_X\le KH_j^{1/D}
\]
for all $j$.
\end{defn}

\begin{thm}[Polynomial envelopes from a scalar invariant]\label{thm:general-polynomial}
Assume that $\Phi_\chi(n)\le C_0n^D+C_0$ for all $n\ge1$, and let
$0<\alpha<D$.  Then the following conclusions hold.
\begin{enumerate}
\item There is a constant $C>0$ such that
\[
  \delta_{\Qcal_\alpha}(n)\le Cn^\alpha+C
\]
for all $n\ge1$.
\item If, in addition, the hard values $H_j=s_{j+1}/2$ are $D$-efficiently
realized by $\chi$, then there exist a constant $c>0$ and a sequence
$n_j\to\infty$ such that
\[
  \delta_{\Qcal_\alpha}(n_j)\ge cn_j^\alpha
\]
for every $j$.
\end{enumerate}
\end{thm}

\begin{proof}
Let $w\in N$ have length at most $n$, and set $A=\chi(w)$.  The word
$wq^{-A}$ has $\chi$-value zero, so a cyclically reduced conjugate of it is
either empty or a defining relator of $\Qcal_\alpha$.  By
\cref{rem:one-cell-whisker}, this contributes at most one cell and reduces the
problem to filling $q^A$.  By
the definition of $\tau_\beta$, the latter word can be filled using at most
$\tau_\beta(A)$ cells of the form $q^{s_j}$ or their inverses.  Hence, using
the polynomial bound on $\Phi_\chi$ and \cref{lem:general-arithmetic},
\[
  \Area_{\Qcal_\alpha}(w)
  \le 1+\tau_\beta(A)
  \le C_1|A|^\beta+C_1
  \le C_2\Phi_\chi(n)^\beta+C_2
  \le C_3n^{D\beta}+C_3.
\]
Since $D\beta=\alpha$, this estimate proves the global upper bound.

Assume now that the hard values $H_j$ are $D$-efficiently realized, and choose
$z_j\in N$ as in \cref{def:efficient-realization}.  Replacing $z_j$ by its
freely reduced representative does not change its element of $F(X)$, does not
change $\chi(z_j)$, and can only decrease its length; hence we may assume that
$z_j$ is freely reduced.  If $\mathcal D_j$ is a van Kampen diagram over
$\Qcal_\alpha$ for $z_j$ with $\nu_j$ cells, then the boundary relation in
$F(X)$ can be written as
\[
  z_j=\prod_{i=1}^{\nu_j} g_i r_i^{\varepsilon_i} g_i^{-1},
  \qquad g_i\in F(X),\quad \varepsilon_i\in\{\pm1\},\quad r_i\in S_\alpha,
\]
up to free reduction.  Applying $\chi$ to this equality expresses $H_j$ as a
signed sum of at most $\nu_j$ elements of $S_\beta$.  Again, cyclic conjugates of relators have the
same $\chi$-value because $\chi$ is conjugacy-invariant.  Relators from
$N\cap\ker\chi$ contribute zero to this scalar equation, so including such cells
cannot reduce the required number of nonzero scalar contributions.  Therefore
\[
  \nu_j\ge \tau_\beta(H_j)\ge c_1H_j^\beta
\]
by \cref{lem:general-arithmetic}.  Since $|z_j|_X\le KH_j^{1/D}$, we have
$H_j^\beta\ge K^{-D\beta}|z_j|_X^{D\beta}=K^{-\alpha}|z_j|_X^\alpha$.  Hence
\[
  \Area_{\Qcal_\alpha}(z_j)\ge c_2|z_j|_X^\alpha.
\]
For these words the lengths $|z_j|_X$ tend to infinity, since there are only
finitely many words of bounded length and $\chi(z_j)=H_j\to\infty$.  The peaks
are therefore asserted at the actual word lengths of the realizing words, namely
at $n_j=|z_j|_X$.  Taking these values of $n_j$ proves the lower-bound peak
assertion.
\end{proof}

\begin{rem}\label{rem:no-pointwise-power-general}
The conclusion of \cref{thm:general-polynomial} is a polynomial envelope with
matching peaks under an efficient-realization hypothesis.  It does not assert
that $\delta_{\Qcal_\alpha}(n)\simeq n^\alpha$ pointwise for all large $n$.
The arithmetic construction is deliberately lacunary, and plateaux may occur
between the hard scales.
\end{rem}

\subsection{Existence of invariants and standard corollaries}

\begin{cor}[Existence of scalar relation invariants]\label{cor:existence-scalar-invariant}
Let
\[
  G=\langle X\mid R\rangle
\]
be a finite presentation.  Let $F(X)$ be the free group on $X$, and let
$N=\ker(F(X)\to G)$.  A conjugacy-invariant homomorphism $\chi:N\to\Z$ with infinite image exists if
and only if $N/[F(X),N]$ has positive free abelian rank.  Moreover,
\[
  \operatorname{rank}_{\Z} N/[F(X),N]
  = \operatorname{rank}_{\Z} H_2(G;\Z)
    + |X| - \operatorname{rank}_{\Z} H_1(G;\Z).
\]
Consequently such an invariant exists whenever
\[
  \operatorname{rank}_{\Z} H_2(G;\Z)+|X|-\operatorname{rank}_{\Z} H_1(G;\Z)>0.
\]
In particular, existence is guaranteed if either $H_2(G;\Z)$ has positive free
abelian rank, or the chosen generating set $X$ has cardinality strictly larger
than the first Betti number of $G$.
\end{cor}

\begin{proof}
A homomorphism $\chi:N\to\Z$ is conjugacy-invariant precisely when it vanishes on
$[F(X),N]$.  Thus such homomorphisms are exactly the homomorphisms
\[
  N/[F(X),N]\to\Z.
\]
One of them has infinite image if and only if $N/[F(X),N]$ has positive
torsion-free rank.

The rank formula follows from the Hopf exact sequence for the presentation
$1\to N\to F(X)\to G\to 1$:
\[
  0\longrightarrow H_2(G;\Z)
  \longrightarrow N/[F(X),N]
  \longrightarrow H_1(F(X);\Z)
  \longrightarrow H_1(G;\Z)
  \longrightarrow 0.
\]
Since $F(X)$ is free on $X$, we have $H_1(F(X);\Z)\cong\Z^{|X|}$.  Taking free
abelian ranks in the exact sequence gives
\[
  \operatorname{rank}_{\Z} N/[F(X),N]
  = \operatorname{rank}_{\Z} H_2(G;\Z)
    + \operatorname{rank}_{\Z}\ker\{H_1(F(X);\Z)\to H_1(G;\Z)\}.
\]
The map $H_1(F(X);\Z)\to H_1(G;\Z)$ is surjective, so the rank of its kernel is
$|X|-\operatorname{rank}_{\Z}H_1(G;\Z)$.  This argument proves the displayed formula and
all the stated consequences.
\end{proof}

\begin{cor}[The linearly bounded case]\label{cor:linear-scalar-invariant}
Suppose that $\chi:N\to\Z$ is a conjugacy-invariant relation invariant with
$\chi(N)=\Z$ and
\[
  |\chi(w)|\le C_0|w|_X+C_0
\]
for every $w\in N$.  Then the following conclusions hold.
\begin{enumerate}
\item The logarithmic construction gives an infinite presentation of $G$ on $X$
with Dehn function $\fsimeq\log n$.
\item For every $0<\alpha<1$, the polynomial construction gives an infinite
presentation of $G$ on $X$ with a global $O(n^\alpha)$ Dehn-function upper
bound and matching $n^\alpha$-order lower-bound peaks along an infinite
sequence.
\end{enumerate}
\end{cor}

\begin{proof}
The hypothesis says that $\Phi_\chi(n)=O(n)$, so \cref{thm:general-log} applies
with $D=1$.  For the polynomial construction, take $D=1$ in
\cref{thm:general-polynomial}.  The hard values are $1$-efficiently realized by
$z_j=q^{H_j}$, because $q$ is cyclically reduced and
\[
  |q^{H_j}|=H_j|q|,
  \qquad \chi(q^{H_j})=H_j.
\]
Thus the lower-bound peak assertion also follows.
\end{proof}

\begin{cor}[Word-hyperbolic groups]\label{cor:hyperbolic-scalar-invariant}
Let $G=\langle X\mid R\rangle$ be a finite presentation of a word-hyperbolic
group, and let $N=\ker(F(X)\to G)$.  Suppose that there exists a
conjugacy-invariant homomorphism $\chi:N\to\Z$ with infinite image.  After
rescaling $\chi$ so that $\chi(N)=\Z$, the following conclusions hold.
\begin{enumerate}
\item The logarithmic construction gives an infinite presentation of $G$ on the
same generating set $X$ with Dehn function $\fsimeq\log n$.
\item For every $0<\alpha<1$, the polynomial construction gives an infinite
presentation of $G$ on $X$ with a global $O(n^\alpha)$ Dehn-function upper
bound and matching $n^\alpha$-order lower-bound peaks along an infinite
sequence.
\end{enumerate}
\end{cor}

\begin{proof}
It remains to check the linear boundedness hypothesis for $\chi$.  Since
$G$ is word-hyperbolic, the finite presentation $\langle X\mid R\rangle$ has a
linear Dehn function, up to the standard equivalence of Dehn functions.  After
changing constants, this gives an actual linear filling estimate for the chosen
finite presentation: there is a constant $C$ such that every word $w\in N$ admits an expression
\[
  w=\prod_{i=1}^m f_i r_i^{\eps_i} f_i^{-1}
\]
in $F(X)$, where $m\le C|w|_X+C$, each $f_i\in F(X)$, each $r_i\in R$, and
$\eps_i\in\{\pm1\}$.  Applying $\chi$ and using conjugacy-invariance gives
\[
  \chi(w)=\sum_{i=1}^m\eps_i\chi(r_i).
\]
Since $R$ is finite,
\[
  |\chi(w)|\le m\max_{r\in R}|\chi(r)|\le C'|w|_X+C'.
\]
Therefore $\chi$ is linearly bounded.  The conclusions follow from
\cref{cor:linear-scalar-invariant}.
\end{proof}

\begin{cor}[The quadratic-area case]\label{cor:quadratic-area-scalar-invariant}
Suppose that $\chi:N\to\Z$ is a conjugacy-invariant relation invariant with
$\chi(N)=\Z$ and
\[
  \Phi_\chi(n)=O(n^2).
\]
Then the following conclusions hold.
\begin{enumerate}
\item The logarithmic construction gives an infinite presentation of $G$ on $X$
with Dehn function $\fsimeq\log n$.
\item If the hard values in \cref{thm:general-polynomial} are
$2$-efficiently realized, then for every $0<\alpha<2$ the polynomial
construction gives a global $O(n^\alpha)$ upper bound and matching
$n^\alpha$-order lower-bound peaks.
\end{enumerate}
\end{cor}

\begin{proof}
The corollary is \cref{thm:general-log,thm:general-polynomial} with $D=2$.
\end{proof}

\begin{rem}
Neither one-endedness nor any boundary-theoretic property is used in the scalar
relation-invariant construction.  Hyperbolicity is used in
\cref{cor:hyperbolic-scalar-invariant} only to guarantee the linear bound
$|\chi(w)|\le C|w|_X$.  More generally, the logarithmic construction only requires
polynomial growth of $\Phi_\chi$, while the polynomial-envelope construction
requires a polynomial upper bound for $\Phi_\chi$ and, for matching peaks, an
efficient-realization hypothesis at the hard scales.
\end{rem}

\section{Applications of the scalar construction}\label{sec:scalar-applications}

This section records the main concrete consequences of
\cref{sec:general-relation-invariant}.  The proofs are now short, because the
hard work has been isolated in the general scalar-invariant theorem.

\subsection{The signed-area invariant for \texorpdfstring{$\Z^2$}{Z2}}

Let
\[
  F=F(a,b),\qquad N=\ker(F\to \Z^2),
\]
where $a$ and $b$ map to the standard basis of $\Z^2$.  If $w\in N$, then $w$
labels a closed edge path in the standard square grid.  Its \emph{signed
algebraic area}, denoted $\Acal(w)$, is the sum of the coefficients in the
unique finite integral $2$-chain in the square grid whose cellular boundary is
that closed path.  Existence follows by assigning to each unit square the
winding number of the path around its interior; only finitely many coefficients
are nonzero.  Uniqueness follows from \cref{lem:no-cycles}.  Equivalently,
\[
  \Acal(w)=\sum_e \epsilon_e x_e,
\]
where the sum is over the vertical edges traversed by the path, $x_e$ is the
$x$-coordinate of the vertical edge, and $\epsilon_e=1$ for an upward traversal
and $\epsilon_e=-1$ for a downward traversal.  This normalization gives
\[
  \Acal([a,b])=1.
\]

\begin{rem}\label{rem:z2-signed-area-cell-count}
The same invariant can also be read directly from van Kampen diagrams over the
standard one-relator presentation
\[
  \langle a,b\mid [a,b]\rangle
\]
of $\Z^2$.  With the normalization above, let $N_+(\mathcal D)$ and $N_-(\mathcal D)$ denote,
respectively, the numbers of $2$-cells of a van Kampen diagram $\mathcal D$ whose
boundary orientations contribute $+1$ and $-1$ to the signed unit-square chain.
Thus a positive cell is one whose boundary is read counterclockwise as
$[a,b]=aba^{-1}b^{-1}$, up to cyclic permutation; a negative cell is one
whose boundary is read clockwise as $[a,b]$, equivalently counterclockwise
as $[a,b]^{-1}=bab^{-1}a^{-1}$, up to cyclic permutation.  Then, for every
such diagram $\mathcal D$ with boundary label $w$,
\[
  \Acal(w)=N_+(\mathcal D)-N_-(\mathcal D).
\]
The displayed formula gives the algebraic, rather than unsigned, count of the
square $2$-cells in the diagram.  In particular, signed area gives the elementary lower bound
$|\Acal(w)|\le \Area_{\langle a,b\mid [a,b]\rangle}(w)$.
\end{rem}

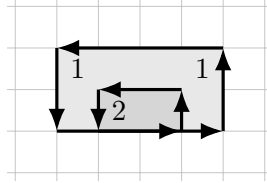
\begin{figure}[htbp]
\centering
\begin{tikzpicture}[scale=.55,>=Latex]
  \draw[step=1,gray!45,very thin] (-.2,-.2) grid (6.2,4.2);
  \fill[gray!18] (1,1) rectangle (5,3);
  \fill[gray!32] (2,1) rectangle (4,2);
  \draw[very thick,->] (1,1) -- (5,1);
  \draw[very thick,->] (5,1) -- (5,3);
  \draw[very thick,->] (5,3) -- (1,3);
  \draw[very thick,->] (1,3) -- (1,1);
  \draw[very thick,->] (2,1) -- (4,1);
  \draw[very thick,->] (4,1) -- (4,2);
  \draw[very thick,->] (4,2) -- (2,2);
  \draw[very thick,->] (2,2) -- (2,1);
  \node at (1.5,2.5) {$1$};
  \node at (2.5,1.5) {$2$};
  \node at (4.5,2.5) {$1$};
  \node[below] at (3,-.75) {signed area = sum of winding numbers over unit squares};
\end{tikzpicture}
\caption{The signed algebraic area of a closed lattice path is the sum of the coefficients of the finite cellular $2$-chain bounded by the path.  Equivalently, it is the sum of the winding numbers over the unit squares.}
\label{fig:signed-area}
\end{figure}

The map $\Acal:N\to\Z$ is a homomorphism and is invariant under conjugation in
$F$.  Additivity follows because concatenating two closed lattice paths adds
their cellular boundary $1$-chains, and hence adds the unique finite cellular
$2$-chains which they bound.  Conjugating a closed path only translates its
associated finite $2$-chain, and translation does not change the sum of the
coefficients.  Hence $\Acal$ is a conjugacy-invariant relation invariant in the
sense of \cref{sec:general-relation-invariant}.  Moreover,
\[
  \Phi_{\Acal}(n)=O(n^2),
\]
because a closed lattice path of length at most $n$ is contained in a square of
side length at most $n$ and therefore has algebraic area bounded by a constant
multiple of $n^2$.

\begin{cor}[Logarithmic presentations of $\Z^2$]\label{cor:z2-log-from-general}
The free abelian group of rank two has an infinite presentation on the
generators $a,b$ with Dehn function $\fsimeq\log n$.
\end{cor}

\begin{proof}
Apply \cref{thm:general-log} to the invariant $\Acal$ and
$q=[a,b]$.  The polynomial bound $\Phi_{\Acal}(n)=O(n^2)$ verifies the
hypothesis of that theorem.  Cyclic conjugates of relators have the same
$\Acal$-value, since $\Acal$ is conjugacy-invariant.
\end{proof}

\begin{cor}[Polynomial envelopes for $\Z^2$]\label{cor:z2-polynomial-from-general}
For each $0<\alpha<2$, let $\Pcal_\alpha$ be the infinite presentation of
$\Z^2$ on the generators $a,b$ constructed in the proof from the signed-area
invariant and the parameter $\alpha$.  Then:
\begin{enumerate}
\item The Dehn function $\delta_{\Pcal_\alpha}$ satisfies
\[
  \delta_{\Pcal_\alpha}(n)\le Cn^\alpha+C
\]
for all $n\ge1$ and some constant $C>0$.  Moreover, there exist a constant
$c>0$ and an infinite sequence $n_i\to\infty$ such that
\[
  \delta_{\Pcal_\alpha}(n_i)\ge cn_i^\alpha
\]
for all $i$.
\item If $0<\alpha_1<\alpha_2<2$, then the Dehn functions of
$\Pcal_{\alpha_1}$ and $\Pcal_{\alpha_2}$ are not finely equivalent.  In
fact,
\[
  \delta_{\Pcal_{\alpha_2}}\not\fpreceq
  \delta_{\Pcal_{\alpha_1}}.
\]
\end{enumerate}
\end{cor}

\begin{proof}
We use \cref{thm:general-polynomial} with $D=2$.  For each $0<\alpha<2$, let
$\Pcal_\alpha=\Qcal_\alpha$ be the presentation supplied by that theorem for
$\chi=\Acal$ and $q=[a,b]$.  It remains to verify the efficient-realization
hypothesis for the hard values $H_j$.  Put
$p_j=\lfloor\sqrt{H_j}\rfloor$ and write
\[
  H_j=p_jq_j+r_j,
  \qquad 0\le r_j<p_j.
\]
The union of the rectangle $[0,p_j]\times[0,q_j]$ with, if $r_j>0$, the strip
$[0,r_j]\times[q_j,q_j+1]$ is a simply connected lattice polyomino of area
$H_j$ and perimeter at most $C\sqrt{H_j}$, for a uniform constant $C$.  Let $z_j$
be its positively oriented boundary word.  Then
\[
  \Acal(z_j)=H_j,
  \qquad |z_j|\le C\sqrt{H_j}.
\]
Thus the hard values are $2$-efficiently realized, and
\cref{thm:general-polynomial} proves the asserted upper bound and lower-bound
peaks in part~(1).

It remains to prove part~(2).  Let $0<\alpha_1<\alpha_2<2$.  By part~(1),
there is a constant $C_1>0$ such that
\[
  \delta_{\Pcal_{\alpha_1}}(n)\le C_1n^{\alpha_1}+C_1
\]
for all $n\ge1$, and there exist a constant $c_2>0$ and an infinite sequence
$n_j\to\infty$ such that
\[
  \delta_{\Pcal_{\alpha_2}}(n_j)\ge c_2n_j^{\alpha_2}.
\]
If $\delta_{\Pcal_{\alpha_2}}\fpreceq\delta_{\Pcal_{\alpha_1}}$, then for
some constant $K\ge1$ one would have
\[
  \delta_{\Pcal_{\alpha_2}}(n)\le
  K\delta_{\Pcal_{\alpha_1}}(Kn+K)+K
\]
for all $n$.  Evaluating this inequality at the sequence $n_j$ gives
\[
  c_2n_j^{\alpha_2}
  \le K\bigl(C_1(Kn_j+K)^{\alpha_1}+C_1\bigr)+K,
\]
which is impossible as $j\to\infty$, since $\alpha_2>\alpha_1$.  Therefore
$\delta_{\Pcal_{\alpha_2}}\not\fpreceq\delta_{\Pcal_{\alpha_1}}$, and the
Dehn functions are not finely equivalent.
\end{proof}

\begin{cor}[Free abelian groups of higher rank]\label{cor:zd-from-general}
Let $d\ge2$.  The free abelian group $\Z^d$, with its standard finite
presentation on generators $x_1,\ldots,x_d$, satisfies the following
conclusions.
\begin{enumerate}
\item It has an infinite presentation on the same generators with Dehn function
$\fsimeq\log n$.
\item For every $0<\alpha<2$, it has an infinite presentation on the same
generators whose Dehn function is globally bounded above by $Cn^\alpha+C$ and
has matching $n^\alpha$-order lower-bound peaks along an infinite sequence.
\end{enumerate}
\end{cor}

\begin{proof}
Let $\pi:F(x_1,\ldots,x_d)\to F(x_1,x_2)$ be the homomorphism which fixes
$x_1,x_2$ and kills $x_3,\ldots,x_d$.  If $w$ lies in the relation subgroup for
$\Z^d$, then $\pi(w)$ lies in the relation subgroup for $\Z^2$.  Composing
$\pi$ with the signed-area invariant for $F(x_1,x_2)\to\Z^2$ gives a scalar
invariant on the relation subgroup for $\Z^d$.  The projection $\pi$ does not
increase word length, so this invariant satisfies the same quadratic bound
$\Phi(n)=O(n^2)$.  It takes the value $1$ on $[x_1,x_2]$.  Moreover, the same
planar polyomino boundary words used in the proof of
\cref{cor:z2-polynomial-from-general} involve only $x_1$ and $x_2$, and
therefore $2$-efficiently realize the hard values in the relation subgroup for
$\Z^d$.  Thus the proofs of
\cref{cor:z2-log-from-general,cor:z2-polynomial-from-general} apply unchanged.
\end{proof}

\begin{rem}
The corollaries in this subsection recover the scalar-area examples for
$\Z^2$.  The same argument also gives the higher-rank extensions stated in
\cref{cor:zd-from-general}.
The earlier dyadic-strip presentation with Dehn function of order $n\log n$ is
not an instance of this scalar construction; it uses a row-wise relation-module
phenomenon rather than only total signed area.
\end{rem}

\subsection{Torsion-free small-cancellation groups}\label{subsec:small-cancellation-examples}

This subsection supplies scalar relation invariants from small-cancellation asphericity.
Let $X$ be finite, and let $R\subseteq F(X)$ be a nonempty finite symmetrized set of nontrivial cyclically reduced words; thus $R$ is closed under cyclic permutation and inversion.  Assume that $R$ satisfies the small-cancellation condition $C'(1/6)$ and that no element of $R$ is a proper power in $F(X)$.
Put
\[
  F=F(X),\qquad N=\ncl_F(R),\qquad G=F/N.
\]
By the standard torsion theorem for small-cancellation groups, these hypotheses
also imply that $G$ is torsion-free.
Choose a finite subset
\[
  R_0=\{r_1,\ldots,r_k\}\subseteq R
\]
which contains exactly one representative from each equivalence class of
members of $R$ under cyclic permutation and inversion.  Thus the symmetrization
of $R_0$ is $R$.  The set $R_0$ is nonredundant in the sense that no two
distinct elements of $R_0$ are conjugate in $F$ to one another or to one
another's inverses.

By the standard asphericity theorem, equivalently the identity theorem, for
$C'(1/6)$ presentations with no proper-power relators, see for example
Lyndon--Schupp \cite[Chapter V]{LyndonSchupp}, the presentation complex of the
finite presentation $\langle X\mid R_0\rangle$ is aspherical.  Here $R_0$ is the
actual finite relator set used in the presentation complex, while $R$ is its
symmetrization.  Equivalently, there are no nontrivial identities among the
relators in $R_0$, apart from the formal Peiffer identities.  More concretely, the cellular chain complex of the universal cover
gives the usual relation-module exact sequence
\[
  0\longrightarrow \pi_2(K)\longrightarrow
  \bigoplus_{i=1}^k \Z G e_i\longrightarrow N/[N,N]\longrightarrow 0,
\]
where $K$ is the presentation complex and $e_i$ corresponds to the relator
$r_i$.  Since $K$ is aspherical, $\pi_2(K)=0$.  Hence the relation module
$N/[N,N]$ is the free $\Z G$-module with basis the images of
$r_1,\ldots,r_k$.  Passing to $G$-coinvariants gives
\[
  N/[F,N]\cong \Z^k,
\]
with basis represented by the images of the relators $r_1,\ldots,r_k$.
Consequently the assignment
\[
  r_i\mapsto 1,
  \qquad i=1,\ldots,k,
\]
extends uniquely to a conjugacy-invariant homomorphism
\[
  \chi_R:N\to\Z.
\]
The image of $\chi_R$ is infinite, since $\chi_R(r_i)=1$ for each $i$.

\begin{rem}\label{rem:small-cancellation-cell-count}
The invariant $\chi_R$ has a diagrammatic interpretation analogous to signed
area.  Let $\mathcal D$ be a van Kampen diagram over the finite presentation
$\langle X\mid R_0\rangle$ with boundary word $w\in N$.  Let
$N_i^+(\mathcal D)$ be the number of $2$-cells whose boundary label is a cyclic
permutation of $r_i$ with positive orientation, and let $N_i^-(\mathcal D)$ be
the number of $2$-cells whose boundary label is a cyclic permutation of
$r_i^{-1}$ with positive orientation, equivalently a cyclic permutation of
$r_i$ with the opposite orientation.  Then
\[
  \chi_R(w)=\sum_{i=1}^k\bigl(N_i^+(\mathcal D)-N_i^-(\mathcal D)\bigr).
\]
Indeed, the boundary relation associated to $\mathcal D$ writes $w$ in $F$ as a
product of conjugates of the relators $r_i^{\pm1}$; applying the
conjugacy-invariant homomorphism $\chi_R$ gives the displayed formula.
\end{rem}

\begin{cor}[Small-cancellation examples]\label{cor:small-cancellation-from-general}
Let $X$ be finite, and let $R\subseteq F(X)$ be a nonempty finite symmetrized
$C'(1/6)$ set of nontrivial cyclically reduced words, none of which is a proper
power.  Put
\[
  G=\langle X\mid R\rangle.
\]
Then the following conclusions hold.
\begin{enumerate}
\item The group $G$ has an infinite presentation on the generating set $X$ with
Dehn function $\fsimeq\log n$.
\item For every $0<\alpha<1$, the group $G$ has an infinite presentation on
$X$ whose Dehn function is globally bounded above by $Cn^\alpha+C$ and has
matching $n^\alpha$-order lower-bound peaks along an infinite sequence.
\end{enumerate}
\end{cor}

\begin{proof}
The invariant $\chi_R:N\to\Z$ constructed above has infinite image.  A finite
$C'(1/6)$ presentation defines a word-hyperbolic group.  Thus \cref{cor:hyperbolic-scalar-invariant} applies to $\chi_R$ and gives the logarithmic conclusion and the polynomial-envelope conclusion for every $0<\alpha<1$.

For the lower-bound peaks in the polynomial-envelope conclusion, the efficient
realization condition is automatic with $D=1$: if $r_i\in R_0$, then
$\chi_R(r_i)=1$, and the word $r_i^M$ has length $M|r_i|$ and
$\chi_R(r_i^M)=M$ for every $M\ge1$.
\end{proof}

\begin{rem}\label{rem:surface-groups-small-cancellation}
The standard presentation of the closed orientable surface group of genus $g\ge2$ is included in \cref{cor:small-cancellation-from-general}.  Indeed, the
relator
\[
  \rho_g=[a_1,b_1]\cdots [a_g,b_g]
\]
has length $4g$, is not a proper power, and its symmetrized set satisfies
$C'(1/6)$.  Indeed, in the cyclic word $\rho_g$ and its inverse no reduced
subword of length two occurs in two distinct cyclic positions among the
symmetrized relators.  Hence every piece has length at most one, while
$1<4g/6$ for $g\ge2$.  Thus the earlier surface-group examples are special
cases of the small-cancellation corollary.
\end{rem}

\subsection{A homological source of examples}

Combining \cref{cor:existence-scalar-invariant} with
\cref{cor:hyperbolic-scalar-invariant} gives the following compact criterion.

\begin{cor}\label{cor:hyperbolic-homological-source}
Let $G=\langle X\mid R\rangle$ be a finite presentation of a word-hyperbolic
group.  If
\[
  \operatorname{rank}_{\Z}H_2(G;\Z)+|X|-\operatorname{rank}_{\Z}H_1(G;\Z)>0,
\]
then the following conclusions hold.
\begin{enumerate}
\item The group $G$ has an infinite presentation on the generating set $X$ with
Dehn function $\fsimeq\log n$.
\item For every $0<\alpha<1$, the group $G$ has an infinite presentation on $X$
with a global $O(n^\alpha)$ Dehn-function upper bound and matching
$n^\alpha$-order lower-bound peaks along an infinite sequence.
\end{enumerate}
\end{cor}

\begin{proof}
The displayed inequality guarantees, by
\cref{cor:existence-scalar-invariant}, a conjugacy-invariant relation invariant
$N\to\Z$ with infinite image.  The group is word-hyperbolic, so
\cref{cor:hyperbolic-scalar-invariant} applies.
\end{proof}

\section{A redundant-generator construction for arbitrary finitely generated groups}\label{sec:arbitrary-fg}

This section proves \cref{thm:intro-arbitrary-fg-continuum}.  The construction is deliberately redundant: add a generator representing the identity and use its exponent sum as a scalar coordinate.  Varying the infinite relator set in this direction reduces filling to the signed coin-counting functions from \cref{lem:general-arithmetic}.

Let $G$ be a finitely generated group and let $Y$ be a finite generating set for
$G$.  Let $t$ be a new letter, put
\[
  X=Y\sqcup\{t\},
\]
and define an epimorphism
\[
  \pi:F(X)\to G
\]
which sends each element of $Y$ to the corresponding generator of $G$ and sends
$t$ to $1\in G$.  Let
\[
  N=\ker \pi.
\]
Let
\[
  \chi:F(X)\to\Z
\]
be the exponent-sum homomorphism in the letter $t$.  We use the same symbol
$\chi$ for its restriction to $N$.  Then $\chi:N\to\Z$ is a conjugacy-invariant
epimorphism, because $t\in N$ and $\chi(t)=1$.

For a subset $S\subseteq \N$ containing $1$, define
\[
  \tau_S(m)=\min\left\{r:\ m=\sum_{i=1}^r \varepsilon_i s_i,
  \ \varepsilon_i\in\{\pm1\},\ s_i\in S\right\}
\]
for $m\in\Z$, with $\tau_S(0)=0$, and put
\[
  T_S(n)=\max\{\tau_S(m): |m|\le n\}.
\]
Thus $T_S$ is the word metric ball-radius function for the group $\Z$ with
respect to the symmetric generating set $\pm S$.

\begin{lem}\label{lem:dummy-generator-coin-comparison}
Let $S\subseteq\N$ contain $1$.  Let $\Rcal_S$ be the set consisting of all
nontrivial cyclically reduced words $u\in N$ with $\chi(u)=0$, together with the
words $t^s$, $s\in S$.  Put
\[
  \Pcal_S=\langle X\mid \Rcal_S\rangle.
\]
Then $\Pcal_S$ presents $G$.  Moreover, for all $n\ge1$,
\[
  T_S(n)\le \delta_{\Pcal_S}(n)\le T_S(n)+1.
\]
In particular, $\delta_{\Pcal_S}\fsimeq T_S$.
\end{lem}

\begin{proof}
First we show that $\Pcal_S$ presents $G$.  Every relator in $\Rcal_S$ belongs
to $N$, so there is a natural epimorphism $G(\Pcal_S)\to G$.  Conversely, let
$w\in N$.  Put $m=\chi(w)$.  Then
\[
  \chi(wt^{-m})=0,
\]
and $wt^{-m}\in N$.  Hence, after cyclic reduction if necessary, $wt^{-m}$ is
either trivial or conjugate to one of the zero-$\chi$ relators in $\Rcal_S$.
Since $1\in S$, the word $t$ is also a defining relator of $\Pcal_S$.  Therefore
$w$ is trivial in $G(\Pcal_S)$.  Thus the kernel of $F(X)\to G(\Pcal_S)$ is
exactly $N$, and $G(\Pcal_S)\cong G$.

Let $w\in N$ be freely reduced and suppose $|w|_X\le n$.  Put $m=\chi(w)$.  Then
$|m|\le n$.  As above, the word $wt^{-m}$ has $\chi$-value zero.  By
\cref{rem:one-cell-whisker}, it can be filled with at most one zero-$\chi$
relator cell.  The remaining word $t^m$ can be filled using a signed coin
expression for $m$: if
\[
  m=\sum_{i=1}^r\varepsilon_i s_i,
  \qquad \varepsilon_i\in\{\pm1\},\quad s_i\in S,
\]
then $t^m$ is filled using $r$ cells with boundary labels $t^{s_i}$ or
$t^{-s_i}$.  Taking a shortest signed expression gives
\[
  \Area_{\Pcal_S}(w)\le 1+\tau_S(m)\le 1+T_S(n).
\]
Thus $\delta_{\Pcal_S}(n)\le T_S(n)+1$.

For the lower bound, take an integer $m$ with $|m|\le n$ and consider the word
$t^m$.  It has length $|m|\le n$ and is null-homotopic over $\Pcal_S$.  Let
$\mathcal D$ be a van Kampen diagram over $\Pcal_S$ with boundary label $t^m$
and with $\nu$ two-cells.  The usual boundary relation in $F(X)$ expresses
$t^m$ as a product of conjugates of defining relators and their inverses.
Applying the exponent-sum homomorphism $\chi$ gives an expression of $m$ as a
signed sum of the elements of $S$, with at most $\nu$ nonzero summands: the
zero-$\chi$ relators contribute $0$, while a cell labelled by $t^s$ or $t^{-s}$
contributes $s$ or $-s$.  Hence
\[
  \nu\ge \tau_S(m).
\]
Since $\mathcal D$ was arbitrary,
\[
  \Area_{\Pcal_S}(t^m)\ge \tau_S(m).
\]
Taking the maximum over all $|m|\le n$ gives
\[
  \delta_{\Pcal_S}(n)\ge T_S(n).
\]
The two inequalities prove the lemma.
\end{proof}

\begin{thm}[A universal continuum construction]\label{thm:arbitrary-fg-continuum}
Let $G$ be any finitely generated group.  Then there exists a finite generating
alphabet $X$ mapping onto $G$ such that varying the relator set in presentations
$\langle X\mid\Rcal\rangle$ of $G$ gives continuum many pairwise distinct fine
equivalence classes of Dehn functions.

More precisely, for every $0<\alpha<1$ there is an infinite presentation
\[
  \Pcal_\alpha=\langle X\mid\Rcal_\alpha\rangle
\]
of $G$ such that
\[
  \delta_{\Pcal_\alpha}(n)\le C_\alpha n^\alpha+C_\alpha
\]
for all $n\ge1$, and there are constants $c_\alpha>0$ and an infinite sequence
$n_j\to\infty$ such that
\[
  \delta_{\Pcal_\alpha}(n_j)\ge c_\alpha n_j^\alpha
\]
for all $j$.  If $0<\alpha<\beta<1$, then
$\delta_{\Pcal_\beta}\not\fpreceq \delta_{\Pcal_\alpha}$.  In particular the
fine equivalence classes obtained in this way are pairwise distinct for
continuum many values of $\alpha$.
\end{thm}

\begin{proof}
Fix a finite generating set $Y$ of $G$ and form the redundant generating
alphabet $X=Y\sqcup\{t\}$ as above, where the new letter $t$ represents the
identity of $G$.  Let $0<\alpha<1$.  Apply the arithmetic construction
of \cref{lem:general-arithmetic} with $\beta=\alpha$.  Thus we obtain a
divisibility sequence
\[
  1=s_0\mid s_1\mid s_2\mid\cdots
\]
and, for $S_\alpha=\{s_j:j\ge0\}$, constants $C_\alpha,c_\alpha>0$ such that
\[
  \tau_{S_\alpha}(m)\le C_\alpha |m|^\alpha+C_\alpha
\]
for every $m\in\Z$, while for the hard values
\[
  H_j=\frac{s_{j+1}}2
\]
one has
\[
  \tau_{S_\alpha}(H_j)\ge c_\alpha H_j^\alpha
\]
for every $j$.  Let
\[
  \Pcal_\alpha=\Pcal_{S_\alpha}
\]
be the presentation from \cref{lem:dummy-generator-coin-comparison}.  The same
lemma gives
\[
  \delta_{\Pcal_\alpha}(n)\le T_{S_\alpha}(n)+1
  \le C'_\alpha n^\alpha+C'_\alpha
\]
for all $n\ge1$.  It also gives, at the lengths $n_j=H_j$,
\[
  \delta_{\Pcal_\alpha}(n_j)
  \ge \Area_{\Pcal_\alpha}(t^{H_j})
  \ge \tau_{S_\alpha}(H_j)
  \ge c_\alpha H_j^\alpha
  =c_\alpha n_j^\alpha.
\]
This proves the upper bound and the lower-bound peaks.

It remains to prove that the fine classes are pairwise distinct.  Let
$0<\alpha<\beta<1$.  Suppose, toward a contradiction, that
\[
  \delta_{\Pcal_\beta}\fpreceq \delta_{\Pcal_\alpha}.
\]
Then there is a constant $K\ge1$ such that
\[
  \delta_{\Pcal_\beta}(n)
  \le K\delta_{\Pcal_\alpha}(Kn+K)+K
\]
for all $n\ge1$.  Along the lower-bound peak sequence $n_j$ for
$\Pcal_\beta$, the upper bound for $\Pcal_\alpha$ gives
\[
  c_\beta n_j^\beta
  \le K(C_\alpha(Kn_j+K)^\alpha+C_\alpha)+K.
\]
The right-hand side is $O(n_j^\alpha)$, while the left-hand side grows like
$n_j^\beta$.  Since $\beta>\alpha$ and $n_j\to\infty$, this is impossible.
Thus $\delta_{\Pcal_\beta}\not\fpreceq \delta_{\Pcal_\alpha}$.  In particular,
no two of the functions arising from different exponents are finely equivalent.
Since the interval $(0,1)$ has the cardinality of the continuum, the theorem
follows.
\end{proof}

\begin{rem}\label{rem:redundant-generator-construction}
The construction in \cref{thm:arbitrary-fg-continuum} is not canonical and does not produce quasi-isometry invariants.  It uses a highly redundant generating set and highly redundant infinite relator sets.  This contrasts with the Dehn spectrum of Osin and Rybak~\cite{OsinRybak}, which allows all null relators up to a length scale and yields quasi-isometry invariants of finitely generated groups.  The point is instead that, for selected infinite presentations, continuum many fine filling-growth classes can occur for every finitely generated group after adding one identity generator.
\end{rem}

\section{Disclosure of AI use}

AI-assisted tools were used during the preparation of preliminary drafts.  The author has checked the arguments, reviewed and edited the text, and takes full responsibility for the content of the paper.

\end{document}